\documentclass[twocolumn]{autart}  

\usepackage{tikz}
\usepackage{graphicx}   
\graphicspath{{./figs/}}      

\usepackage{mydef} 

\begin{document}

\begin{frontmatter}

\title{Distributionally \, Robust \, Nash \, Equilibrium \, Seeking \, with \, Partial \, Observations \, and \, Distributed \, Communication\thanksref{footnoteinfo}} 

\thanks[footnoteinfo]{This paper was not presented at any IFAC 
meeting. Corresponding author N.~Mandal.}

\author[dept]{Nirabhra Mandal}\ead{nmandal@ucsd.edu},  \hspace{2ex}
\author[dept]{Sonia Mart{\'\i}nez}\ead{soniamd@ucsd.edu}             

\address[dept]{Mechanical \& Aerospace Engineering Department,
  University of California San
  Diego.}  

\begin{abstract}                      
  In this work, we study stochastic one-shot games where agents'
  utilities depend on the collective strategy profiles of other agents
  as well as on some well-behaved randomness. While each
  decision-maker is agnostic to the random variable's underlying
  distribution, they have access to finitely many i.i.d. samples
  generated from it. We consider two cases: one where samples are
  shared; and another, more special one, where samples are
  individually accessible. To hedge against the unknown uncertainty,
  each agent plays a distributionally robust game and aims to
  maximize the worst-case expected utility over a Wasserstein ball
  around the sample average distribution. In this setting, we provide
  conditions under which the game has a non-empty set of
  distributionally robust Nash equilibria (DRoNE) and then
  characterize the closeness of the DRoNE set to the Nash equilibria
  (NE) of the associated stochastic game. We then propose an inertial,
  supported, better response, ascending supergradient dynamics
  (\algoname) that seeks the DRoNE's when the distributionally robust
  game possesses what we term as amicable supergradients. This forms
  the basis of a distributed version (d-\algoname) where agents
  estimate others' strategies by means of a dynamic consensus
  subroutine over a directed communication network. While initially
  the distributed algorithm works in the case where agents have
  individual samples, we later extend this to the case of shared
  observations under certain simplifying assumptions. This involves
  analyzing a tractable reformulation of the distributionally robust
  optimization problem and solving it in a distributed manner to
  compute the required supergradients. Simulations illustrate our
  results.
\end{abstract}

\end{frontmatter}
\section{Introduction}\label{sec:intro}

The design of complex interacting systems via reward-oriented agents
is transforming real-world applications in autonomous systems,
critical infrastructure,
and mixed human-robot teams. 
Strategic settings, where emerging multi-agent behaviors depend on the
selfish actions of its members, can greatly be degraded by unknown
environmental uncertainty, potentially furthering their inherent
inefficiencies. With broad and distributed accessibility to data, it
is now possible to construct more powerful agents that exploit new
techniques such as distributionally robust optimization
(DRO). Although this can result in more robust equilibria, the
question remains as to how to effectively leverage this framework to
design provably-correct distributed learning algorithms for
multi-agent networks. Motivated by this, we study here the integration
of DRO with the game-theoretic framework, and propose novel learning
algorithms that can practically converge to an associated equilibrium.

\paragraph*{Literature review}

Non-cooperative Game Theory~\cite{TB:10a,YN:14}, 
provides a powerful framework for the analysis and design of
engineered multi-agent systems. Recent applications include
multi-vehicle networks~\cite{JRM-GA-JSS:09,MZ-SM:10},
telecommunication~\cite{YP-LP:09} and cyber-physical
systems~\cite{WS-ZH-HVP-TB:12}.
In particular, the consideration of self-interested agents allows for
the modeling of human behavior~\cite{JG-RS:14,PW-NM-SM-PT:24-tcns},
and the evaluation of their impact in socio-technical
systems~\cite{QZ-TB:25}, which opens the door to incentive
design~\cite{TR-IT:19,SH-BT-MA:24}.


In this framework, the concept of Nash equilibria (NE) characterizes
those lucrative equilibrium strategies that rational agents adopt in
response to others' decisions.  Thus, a variety of Nash-seeking
algorithms;
\emph{e.g.},~\cite{JSS-GA:05,PF-MK-TB:10,MY-GH:17,MY-QH-LD-SX:23},
have been developed to learn such equilibria. Often, these algorithms
require strong assumptions on utility functions and their
derivatives. In practice, only imperfect information about other
agents’ strategies is available~\cite{MY-QH-LD-SX:23}. While the early
work~\cite{JSS-GA:05} addresses this lack of information by means of
fictitious play~\cite{JSS-GA:05}, this requires a complete
communication graph among the agents.  Subsequently, this lack of
information has been mostly addressed via
consensus~\cite{JK-AN-UVS:12,MY-GH:17}, and gossip-based
algorithms~\cite{FS-LP:16}, which borrow ideas from the literature on
distributed optimization~\cite{AN-AO-PAP:08} and dynamic
consensus~\cite{SSK-BVS-JC-RAF-KML-SM:18-csm}. When utilities are
differentiable, these Nash-seeking algorithms can be interpreted as
passivity-based designs~\cite{DG-LP:18}~\cite{CDP-SG:19}, where the
gradient play of agents are corrected using an estimation term.  In
general, equilibria associated with games with ``convex'' utilities
over continuous strategy sets can be characterized via variational
inequalities~\cite{AK-UVS:10,RTR:18}, which has been exploited in the
development of Nash-seeking
algorithms~\cite{AK-UVS:10,YWC-CK-MA:24}. In
particular,~\cite{AK-UVS:10} uses proximal operators as regularization
terms similarly to how~\cite{TC-IK:23} use proximal terms for bi-level
gradient descent. We refer the reader to~\cite{MY-QH-LD-SX:23} for an
extensive list of distributed Nash-seeking algorithms.

Optimization under uncertainty is studied extensively in the field of
stochastic programming~\cite{AS-DD-AR:14} and its robust
counterpart~\cite{AB-LEG-AN:09}. However, these setups often require
knowledge of the underlying probability distribution, or
over-approximating bounds on the support of this distribution. When
many samples are available, deterministic optimization algorithms
using sample average approximations~\cite{SK-RP-SGH:14} provide
appropriate results.  For limited samples, distributionally robust
optimization (DRO) formulations, where one tries to maximize the
worst-case expected reward over an ambiguity set related to the
unknown data-generating distribution, can provide decisions with a
high confidence performance. These problems have been studied
since~\cite{SHE-KJA-SK:57} and, more recently
in~\cite{DB-MS:04,PME-DK:17}, and~\cite{CQ-HJ-PY:25}.  When dealing
with continuous random variables, DRO problems translate in an
infinite-dimensional optimization. As such,~\cite{PME-DK:17} leverages
concentration of measure inequalities to provide tractable
finite-dimensional reformulations under certain conditions, and
out-of-sample guarantees.  This form of Wasserstein DRO is extended to
the case of streaming samples in~\cite{DL-SM:19-tac-online} and to the
case where samples are obtained from a dynamical system under partial
observations~\cite{DB-JC-SM:23-tac}. In this article, we apply results
from DRO over a continuous sample space with fixed number of samples.

Distributional robust game
theory~\cite{SQ-DM-YZ-YD:17,YL-HX-SSY-JZ:18,JG-HT:18} is receiving
increasing attention, with current research differing in their
treatment of uncertainty. The work~\cite{SQ-DM-YZ-YD:17} focuses on
discrete random variables,
while~\cite{YL-HX-SSY-JZ:18,JG-HT:18,GP-TZ-QZ:20} handle continuous
ones. In these studies every agent is affected by their own
uncertainty, and neither account for the effect of limited
observations from shared samples among players in the equilibrium
formation.  More recently,~\cite{GP-TZ-QZ:20} produces a tractable
reformulation of a distributional robust game for finite strategy
spaces and mixed strategy equilibria.  Further,~\cite{GP-TZ-QZ:20}
shows that, as the number of data samples increases to infinity, the
distributionally robust equilibria asymptotically approach those of
the stochastic game. However, no out-of-sample guarantees under a
finite number of data points is provided.  A similar result is
provided in~\cite{FF-BF:23} for generalized Nash games. However, here
the uncertainty is used to define chance constraints on the the
feasible set of strategies, which does not directly affect the
utilities of the agents. None of these works aim to solve the robust
equilibrium seeking problem in a distributed manner.  The
work~\cite{ZW-GP-SG-MMZ-KHJ:25} proposes a distributionally robust
Nash equilibrium seeking algorithm. However, it assumes
differentiability of the associated utilities to produce a monotone
variational inequality reformulation that facilitates a
pseudo-gradient update towards the Nash equilibria.  On a different
take, the work in~\cite{NL-SF-SB-FD-DP:25} rather assumes that it is
the opponents' strategies that are generated by an unknown
distribution, and study an associated strategically robust
game. Distributionally robust Stackleberg games are considered
in~\cite{LL-SL-HX-DEQ:25}, where the authors assume strong Lipschitz
assumptions on the utility functions in order to provide gradient
based updates for the leader-follower dynamic game.  In previous
work~\cite{NM-AS-CN-SM:25}, we studied a distributionally robust task
allocation game among self-interested agents. However these results
apply for specific utility functions that model task allocation and
that are linear in the random variables.

\paragraph*{Contributions}
This paper considers a distributionally robust game (DRG) among a
group of strategic agents. In this setting, the distributionally robust
utilities are in general non-smooth.  Unlike previous
works in the literature, 
agents construct their ambiguity sets from partially observed samples
obtained as the image of an observation function. 

Assuming that these measurements can be stacked to define an
invertible mapping, we propose a modification of the ambiguity ball
that contains the original distribution with high confidence. Roughly
speaking, this approximation requires that the observation mapping
does not deviate too much from an optimal transport of the unknown
distribution to the empirical one.  We then identify conditions that
guarantee the existence of NE for the DRG and the associated
stochastic game. These conditions, which are standard requirements for
NE of continuous games, generalize well to DRGs. In particular,
the aforementioned uncertainty bounds allow us to obtain a
finite-sample error bound between the distributionally robust NE
(DRoNE) and the NE of the stochastic game. 
As robust utilities may be non-smooth, our DRoNE seeking strategy
leverages supergradients and inertial terms to produce an algorithm
(\algoname) that replicates monotonicity-like behavior for the
game. The algorithm is equally useful for seeking NE for any
non-smooth utility game when the game possesses what we term
``amicable'' supergradients. After this, we provide ultimate
boundedness guarantees to an arbitrarily small region around the DRoNE
set. Moreover, we report a heuristic to tune the algorithm parameters
to attain aforementioned arbitrarily small region to which the
algorithm converges to. We then exploit input-to-state stability
properties of dynamic average consensus algorithms and make \algoname
distributed by allowing agents to estimate others' strategies over a
communication network. Finally, to tackle shared observations, we
adapt the tractable reformulation of the distributionally robust
optimization into a distributed optimization problem. This allows
agents to compute their supergradients in a fully distributed
manner. As such, d-\algoname for shared observations does not require
full information from all agents and can handle the effect of partial
observations shared among the group of agents. We present simulations
that support our results.

The rest of the paper is organized as follows. In
Section~\ref{sec:prelim}, we list our notations and pre-existing
results that we use in this article. In Section~\ref{sec:problem}, we
formalize the distributionally robust game and state the problems we
are interested in addressing. Then in Section~\ref{sec:ne}, we relate
the Nash equilibria of the distributionally robust game with the
underlying stochastic game and provide conditions under which the
aforementioned equilibria exist. In Section~\ref{sec:algo}, we provide
an algorithm that seeks out the Nash equlibria of the distributionally
robust game. This algorithm requires central information from all
agents, so in the next two sections we provide ways to make the
algorithm distributed. Specifically, in Section~\ref{sec:dist_algo} we
provide a distributed algorithm under the simplifying assumption that
the agents are able to compute their own strategy updates provided
they have an estimate of other agents' strategies. In
Section~\ref{sec:dist_drones_shared} we provide a distributed
algorithm under the full generality where agents have to estimate
their own strategy updates and other agents' strategies over a
communication network. Finally, we provide simulations demonstrating
our results in Section~\ref{sec:sims} and conclude the article in
Section~\ref{sec:conclude}. We have included some of the proofs in the
Appendix.

\section{Preliminaries} \label{sec:prelim} Here, we formalize the
notations and briefly list some well-known concepts that are used to
solve the problem formulated in the following section.

\subsection{Notations}
The sets of real numbers, non-negative real numbers, positive real
numbers, non-negative integers, and positive integers are denoted as
$\real$, $\realnonneg$, $\realpos$, $\intnonneg$, and $\intpos$
respectively. For a set $\set$, $|\set|$ denotes its cardinality, and
$\set^n$ the $n$ Cartesian product of $\set$ with itself.  For a
metric space $(\mathcal{X},d)$,
$\ball_\epsilon(x) \ldef \{y \in \Xset | d(x,y) \leq \epsilon\}$ is
the $\epsilon$-radius closed ball around $x \in \Xset$.  Given
$\x \in \real^n$, and $\set \subseteq \real^n$; we use $\|\x\|_p$ to
denote the $p$-norm of $\x$, with $p \ge 1$, and
$d(\x,\set) \ldef \inf_{\y \in \set} \|\x-\y\|$ for the distance of
the vector to the set. Without any subscript, $\|\x\|$ refers to the
$2$-norm of $\x \in \real^n$. With appropriate dimensions, $\zero$ is
the zero vector, $\one$ is the vector with all entries $1$, $\bld{I}$
is the identity matrix, and $\bld{O}$ is the zero matrix.  For two
sets $\set_1$ and $\set_2$,
$\set_1 + \set_2 \ldef \{s_1+s_2 \,|\, s_1 \in \set_1, s_2 \in
\set_2\}$ denotes its Minkowski sum.  The probability of an event
(over a measurable space, which should be clear from the context) is
denoted by $\Pr \{\cdot \}$.  For a measurable function
$f : \real^n \to [0,\infty)$, $\|f\|$ denotes the $L^2$-norm of the
function.  Lastly, the empty set is denoted as $\emptyset$.

\subsection{Measure theory and concentration results}\label{subsec:measure_th}
Suppose $\Fcol$ is a $\sigma$-algebra on $\real^m$; 
and let $\mu$ and $\nu$ be measures on the measure space
$(\real^m, \Fcol)$. Then, $\nu$ is said to be absolutely continuous
w.r.t. $\mu$ (denoted as $\nu \ll \mu$) if $\mu(\set) = 0$ implies
that $\nu(\set) = 0$. 
Suppose that $f:\real^m \to \real^m$ is
a measurable function. Then, the \emph{$f$-pushforward measure} of
$\mu$ is a new measure denoted by $f_\sharp\mu$; and defined as
$f_\sharp\mu(\set) = \mu(f^{-1}(\set))$, $\forall \set \in \Fcol$.
Finally, we denote the Dirac delta function at $\x$ using $\delta_\x$;
and the multivariate Gaussian distribution with mean $\x$ and
covariance $\Sigma$ as $\distgauss{\x}{\Sigma}$.

Assume that $\xi \in \real^m$ is a random variable with an unknown
probability distribution, $\distp^\star$. 
When $\distp^\star$ is supported on
$\Xi \subseteq \real^m$ and $\exists \, a > 1$ such that
$\expect[\xi \sim \distp^\star]{\exp(\|\xi\|^a)} < \infty$; we say
that~$\distp^\star$ is light-tailed.  Let $\mathcal{M}(\Xi)$ be the
set of all probability distributions $\mathbb{Q}$ on $\Xi$ with
bounded first-order and second-order moments.
Then, the $1$-Wasserstein distance (we drop the $1$ in the sequel for
brevity) between two distributions
$\distq_1, \distq_2 \in \mathcal{M}(\Xi)$ is defined as~\cite{FS:15}
\begin{align}
  & d_\mathsf{W}(\distq_1, \distq_2)
    \ldef \hspace*{-2ex}\inf_{\pi \in \Pi(\distq_1,\distq_2)} \int_{\Xi^2}
    \|\xi_1 - \xi_2\|_1 \, \pi(\dg\xi_1, \dg\xi_2)\,, 
	\label{eq:wasserstein}
\end{align}
where, $\Pi(\distq_1,\distq_2)$ is the set of joint probability
distributions of $\xi_1$ and $\xi_2$ with marginals $\distq_1$ and
$\distq_2$ respectively. Further,
\begin{align*}
  \mathcal{B}_\varepsilon (\distp)
  \ldef \{ \distq \in \mathcal{M}(\Xi) \,|\,
  d_\mathsf{W}(\distq, \distp) \leq \varepsilon \},
\end{align*}
is the closed Wasserstein ball of radius $\varepsilon$ around
$\distp \in \mathcal{M}(\Xi)$. Under some conditions, the
aforementioned Wasserstein distance can be written in different
equivalent ways, which become useful later. When $\bbq_1$ is
absolutely continuous w.r.t. the Lebesgue measure, the Monge
formulation of optimal transport gives,
  \begin{align} 
  \dist[W](\bbq_1,\bbq_2) & = \inf \Big\{
      \int_{\Xi} \|\xi - T(\xi)\|_1 \, \bbq_1(\dg \xi) \Big|
      T_{\sharp}\bbq_1 = \bbq_2 \Big\},
    		\label{eq:monge_dist}
  \end{align}
  where $T : \real^m \to \real^m$ is a measurable function, or
  deterministic transport map.  Finally, the Kantorovich-Rubinstein
  description of the Wasserstein distance is,
\begin{align}
  \dist[W](\bbq_1,\bbq_2)
  = \sup_{f \in \mathcal{L}}
  \Big\{ \int_{\Xi} f(\xi) \bbq_1(\dg \xi)
  - \int_{\Xi} f(\xi) \bbq_2(\dg \xi) \Big\},
	\label{eq:kr_dist}
\end{align}
where
$\mathcal{L} \ldef \big\{f \,\,\big|\,\, |f(\xi_1) - f(\xi_2)| \leq \|\xi_1 -
\xi_2\|,\, \forall \xi_1,\xi_2 \in \Xi \big\}$ is the set of all Lipschitz
continuous functions with Lipschitz constant $1$.

Let $\{\hat{\xi}(k) \}_{k = 1}^N \subseteq \real^m$ be i.i.d.~samples
generated by a distribution $\distp^\star$.  The following result
characterizes the closeness of $\distp^\star$ to the sample average
distribution
$\distphat^{N} \ldef \frac{1}{N} \sum_{k = 1}^{N}
\delta_{\hat{\xi}(k)}$.  The result is useful to provide guarantees
from \emph{distributionally robust optimization}~\cite{PME-DK:17}.

\begin{thm}
\label{thm:dro_bound}
\cite{DB-JC-SM:19-ecc} Suppose $\theta \in (0,1)$, and
\begin{align}
	\varepsilon = 
	\begin{cases}
          \left( \frac{\log(c_1 \theta^{-1})}{c_2 N} \right)^{1/\max\{m,2\}},
          & \text{if } N \geq \frac{\log(c_1 \theta^{-1})}{c_2}\,; \\
          \left( \frac{\log(c_1 \theta^{-1})}{c_2 N} \right)^{1/a},
          & \text{if } N < \frac{\log(c_1 \theta^{-1})}{c_2} \,;
	\end{cases}
	\label{eq:dro_bound}
\end{align}
where $c_1$ , $c_2$ are positive constants that only depend on $a$ and
$m$.  Then,
$ \Pr \big\{ \distp^\star \in \mathcal{B}_\varepsilon (\distphat^{N})
\big\} \geq 1 - \theta$.  \bulletend
\end{thm}


\subsection{Size functions and input-to-state stability}
We recall that a continuous function
$\sigma : \realnonneg \to \realnonneg$ is class-$\Mcal{K}$
(\emph{i.e.}
$\sigma \in \Mcal{K}$) if it is strictly increasing and
$\sigma(0) = 0$. Further, $\sigma$ is class-$\Mcal{K}_\infty$
(\emph{i.e.} $\sigma \in \Mcal{K}_\infty$) if $\sigma \in \Mcal{K}$
and $\sigma(x) \to \infty$ as $x \to \infty$. Finally, a continuous
function $\beta : \realnonneg \times \realnonneg \to \realnonneg$ is
class-$\Mcal{KL}$ (\emph{i.e.} $\beta \in \Mcal{KL}$) if
$\forall y \in \realnonneg$, $\beta (\cdot,y) \in \Mcal{K}$ and
$\forall x \in \realnonneg$, $\beta(x,\cdot)$ is decreasing with
$\beta(x,y) \to 0$ as $y \to \infty$. The following notion will be
used in the stability analysis in the sequel.
 
\begin{defn}\thmtitle{Size function~\cite{ES:22}}
\label{def:size_func}
A function ${\omega: \real^n \to \real}$ is said to be a size function
for the compact subset $\set \subset \real^n$ if $\omega$ is i)
continuous; ii) positive definite with respect to $\set$, \emph{i.e.}
$\omega(\x) = 0$, if $\x \in \set$ and $\omega(\x) > 0$, if
$\x \notin \set$; and iii) proper, \emph{i.e.}
$\forall r \in \realnonneg$, the sublevel set
$\{\x\,|\, \omega(\x) \leq r \}$ is a compact subset of $\real^n$.
\bulletend
\end{defn}
Size functions have also been referred to as \emph{proper indicators}
in~\cite{CMK-PMD:12,NN-RG-LG-BSR-FRW:17}. For example, the previously
defined distance function $d(\cdot,\set)$ is a size function for
$\set$. With this, the following result provides a useful comparison
between size functions and class-$\Mcal{K}_\infty$ functions.
\begin{lem}
\label{lem:size_func}
\cite{ES:22} Suppose $\omega_1, \omega_2$ are size functions for
$\set \subset \real^n$. Then there exists functions
$\sigma_1, \sigma_2 \in \Mcal{K}_\infty$ such that
\begin{align*}
  \sigma_1(\omega_2(\x)) \leq \omega_1(\x) \leq \sigma_2(\omega_2(\x)),
  \quad \forall \x \in \real^n. \tag*{\bulletsym}
\end{align*}
\end{lem}

The previously discussed concepts are used to recall the notion of
\emph{input-to-state stability} for a non-linear autonomous system
\begin{align}
	\x(t+1) = \f(\x(t),\u(t)), \quad \forall t \in \intnonneg;
	\label{eq:dummy_sys}
\end{align}
with continuous $\f: \real^n \times \real^m \to \real^n$.
\begin{defn}\thmtitle{Input-to-state stability~\cite{EDS-YW:99}}
\label{def:iss}
Suppose $\omega$ is a size 
function for
$\set \subset \real^n$. Then,~\eqref{eq:dummy_sys} is said to be
\emph{input-to-state stable (ISS)} if there exists
$\beta \in \Mcal{KL}$ and $\gamma \in \Mcal{K}$, such that
$\forall \x_0 \in \real^n$ and $\forall \{\u(t)\}_{t \in \intnonneg}$
(bounded), the solution $\{\x(t)\}_{t \in \intnonneg}$ satisfies,
$\forall t \in \intnonneg$,
\begin{align*}
  \omega(\x(t)) \leq \max
  \left\{\beta\Big(\omega(\x_0),t\Big), \,\,
  \gamma\left(\sup_{\tau \in [0,t]} \|\u(\tau)\|\right)\right\}\,.
  \tag*{\bulletsym}
\end{align*}
\end{defn}
Finally,~\eqref{eq:dummy_sys} is
ISS~\cite{NN-RG-LG-BSR-FRW:17}~\cite{ES:22} if and only if there
exists a continuous (Lyapunov) function $V : \real^n \to \real$;
$\sigma_1, \sigma_2, \sigma_3 \in \Mcal{K}_\infty$; and
$\gamma \in \Mcal{K}$ such that, $\forall \x \in \real^n$,
$\forall \u \in \real^m$,
\begin{subequations}
\begin{align}
  \label{eq:iss_v_bound} &
                           \sigma_1(\omega(\x)) \leq V(\x)
                           \leq \sigma_2(\omega(\x)), \\
  \label{eq:iss_v_change} &
                            V(\f(\x,\u)) - V(\x) \leq -
                            \sigma_3(V(\x)) + \gamma(\|\u\|).
\end{align}
\label{eq:iss}
\end{subequations}  

\subsection{Convex analysis}
A set $\dom \in \real^n$ is said to be convex if for any
$\x,\y\in \dom$, $\lambda\x +[1-\lambda]\y \in \dom$,
$\forall \, \lambda \in [0,1]$. For a convex set $\dom$ and a point
$\x \in \dom$,
$\nc_{\dom}(\x) \ldef \{\v \in \real^n \,|\, \v^\top(\y - \x) \leq 0,
\forall \, \y \in \dom\}$ is the normal cone at $\x$ to the set
$\dom$.

A function $f : \dom \to \real$ with a convex domain
$\dom \subseteq \real^n$ is 
convex if
$f(\lambda\x + [1-\lambda]\y) \leq \lambda f(\x) + [1-\lambda] f(\y)$,
$\forall \x,\y \in \dom$, $\forall \lambda \in [0,1]$. Further, $f$ is
concave iff $-f$ is convex. The concavity of functions can be
characterized via \emph{supergradients}, as defined next.

\begin{defn}
\label{def:subgr}
\thmtitle{Supergradient of a concave function} Consider the concave
function $f$ defined earlier. Then $\zetab \in \real^n$ is said to be a
supergradient of $f$ at $\x$ if
\begin{align}
	f(\y) \leq f(\x) + \ip{\zetab}{[\y-\x]}, \forall \y \in \dom\,.
\end{align}
The set of all supergradients of $f$ at $\x$ is $\del f(\x)$.
\bulletend
\end{defn}
For such a concave function, $\del f(\x)$ is always a closed convex set.
Consider a function $f:\dom_1\times\dom_2 \to \real$ (with
$\dom_1 \subseteq \real^n$ convex) for which $f(\cdot,\y)$ is concave,
$\forall \y \in \dom_2$. Then, we use $\del_\x [f(\x,\y)] |_{\xb}$ to
denote the set of supergradients of $f(\cdot,\y)$ at $\xb \in
\dom_1$. The equivalent notion for a convex function is referred to as
a \emph{subgradient}.

\subsection{Graph theory}
A \emph{directed graph}~\cite{RD:17}, $\grph \ldef (\agt,\edg)$, is a
tuple that has the following elements,
\begin{enumerate}[label=\alph*)]
\item a set of \emph{nodes} (here agents $\agt$); and
\item a set of \emph{arcs} $\edg \subseteq \agt \times \agt$
  between the nodes.
\end{enumerate}
The set $\neigh_i \ldef \{ j \in \agt \,|\, (j,i) \in \edg\}$ denotes
the (in) \emph{neighbors} of node $i \in \agt$ and
$\neighb_i \ldef \neigh_i \cup \{i\}$. A \emph{path} is an ordered set
of non-repeating nodes such that each tuple of adjacent nodes belongs
to $\edg$. The graph $\grph$ is said to be strongly connected if there
exists a path from every node to every other
node. 

\section{Problem Formulation} \label{sec:problem} 
We recall that a \emph{strategic form game}~\cite{YN:14} is a tuple
\begin{align*}
  \game \ldef \left<\agt,\{\S_i \}_{i\in \agt}, \{U_i\}_{i \in \agt} \right>,
\end{align*}
consisting of a set of \emph{players} (or \emph{agents})
$\agt \ldef \{1,\cdots,n\}$; a set of \emph{strategies}
$\s_i \in \S_i \subseteq \real^{n_i}$ (with components $s_{i,l}$)
available to each $i \in \agt$; and a set of \emph{utility functions}
$U_i : \times_{j \in \agt}\S_j \to \real$ over the strategy profiles
of all the agents. We assume that each agent is selfish and interested
in maximizing its own utility.  In what follows, $\s_{-i}$ denotes the
strategy profile of all players other than $i \in \agt$. For the sake
of brevity, we denote $\set \ldef \times_{i \in \agt} \set_i$ the set
of combined strategies, and $\forall i \in \agt$,
$\set_{-i} \ldef \times_{j \in \agt \setminus \{i\}} \set_j$ the set
of strategy profiles of all players other than $i \in \agt$. Moreover,
$\s \ldef [\s_1^\top,\cdots,\s_n^\top]^\top$ denotes the stacked
vector of all strategies. A well-known concept that captures the
equilibrium behavior of such a group of agents trying to selfishly
maximize their own utility is the Nash equilibrium (NE). Next, we
formally state the definition of an $\eta$-NE of a strategic form
game.

\begin{defn}
\label{def:eps_ne}
\thmtitle{$\eta$-Nash equilibrium} A strategy
$(\sst_i,\sst_{-i})$ is said to be an \emph{$\eta$-Nash
  equilibrium} (with $\eta \geq 0$) of $\game$ iff
\begin{align*}
  U_i(\sst_i,\sst_{-i}) + \eta \geq U_i(\s_i,\sst_{-i}), \,
  \forall \s_i \in \set_i, \, \forall i \in \agt\,. 
\end{align*}
The set of $\eta$-Nash equilibria ($\eta$-NE) of $\game$ is given by
$\ne_\eta(\game)$. If $\eta = 0$, this set becomes the set of Nash
equilibria (NE), $\ne(\game)$, of the game.  \bulletend
\end{defn}

In this paper, we focus on a case where every agent's utility is
affected by some random variable, which can be caused by noisy
measurements or the environment. Specifically, suppose agent
$i \in \agt$ obtains a utility of $U_i(\s_i,\s_{-i};\xi)$, where
$\xi \in \Xi \subseteq \real^m$ is a random variable.  If the
distribution of $\xi$ were known (say $\xi \sim \distp^\star$), each
agent could maximize their expected utility given as 
\begin{align}
  \U[S]_i(\s_i,\s_{-i}) \ldef \expect[\xi \sim
  \distp^\star]{U_i(\s_i,\s_{-i};\xi)}\,.
	\label{eq:utility_s}
\end{align}
Here, we assume that $\mathbb{P}^\star \in \mathcal{M}(\Xi)$ with
$\Xi \subset \real^m$ a compact subset. Moreover we assume that
$\distp^\star$ is absolutely continuous w.r.t the Lebesgue measure and
$\distp^\star$ is light-tailed. This results into the \emph{stochastic
  game}
\begin{align}
  \game[S]
  \ldef {\left<\agt,\{\set_i\}_{i \in \agt},\{\U[S]_i\}_{i \in \agt} \right>}.
	\label{eq:game_s}
\end{align}
We make the following assumptions throughout the paper regarding the
utilities and the strategies of agents.
\begin{sassum}\label{asmp:gen_game_assumption}
  The components of $\game[S]$ in~\eqref{eq:game_s} satisfy the
  following properties for each $i \in \agt$;
\begin{enumerate}
\item \label{asmp:strategy_compact} $\set_i \subseteq \real^{n_i}$ is
  non-empty, convex, and compact with diameter
  $D_i \in \real_{\geq 0}$;
\item \label{asmp:utility_concave} $\U_i$ is continuous and
  $\U_i(\s_i,\s_{-i};\xi)$ is concave in $\s_i$,
  $\forall \, \s_{-i}, \xi$; and
\item \label{asmp:utility_xi_lipschitz} $\U_i(\s_i,\s_{-i};\xi)$ is
  $L_i$-Lipschitz in $\xi$, $\forall \, \s_i, \s_{-i}$.  \bulletend
\end{enumerate}  
\end{sassum}
In particular, we will assume that the original distribution
$\distp^\star$ is unknown to the agents, yet they will be able to
sample from it. However, such samples may correspond to partial
measurements, as we state next.
\begin{assum}\thmtitle{Shared i.i.d. samples}
  \label{asmp:shared_sample}
  Each $i \in \agt$ has access to
  $\{\h_i(\xih^{(1)}),\cdots, \h_i(\xih^{(N)})\}$, a set of
  observations from $\{\xih^{(1)},\cdots, \xih^{(N)}\}$ i.i.d.~samples
  of~$\bbp^\star$.  \bulletend
\end{assum}

The previous assumption applies to agents endowed with heterogeneous
sensors and/or computation capabilities. For the sake of brevity, we
define
\begin{align*}
	\h \ldef [\h_1^\top, \cdots, \h_n^\top]^\top \, ,
\end{align*}
as the stacked vector of functions that take in the samples as an
input and return the stacked observations of all the agents as an
output; \emph{i.e.}
$\h(\xih^{(k)}) \ldef [\h_1(\xih^{(k)})^\top, \cdots,
\h_n(\xih^{(k)})^\top]^\top$, $\forall k \in
\{1,\cdots,N\}$.

Further,
we assume the following on the function $\h$.
\begin{assum}\thmtitle{Observability}
  \label{asmp:obs_continuous}
  The observation function $\h : \real^m \to \real^m$ is invertible
  and measurable.  \bulletend
\end{assum}
The previous assumption requires that there exists a sufficiently
large number of agents so that $\h$ is invertible. This requirement
could be mitigated by means of an observer when agents are able to
obtain a continuous stream of samples by
themselves~\cite{DL-SM:19-tac-online} or from a dynamical
system~\cite{DB-JC-SM:23-tac}. We leave these considerations for
future work. 

To deal with the unknown distribution, an ambiguity set centered at
the average sample distribution
\begin{align*}
	\distphat_{\h}^{N} \ldef \frac{1}{N} \sum_{k = 1}^{N}
	 \delta_{\h(\xih^{(k)})}\,,
\end{align*}
is constructed. Recall that $\delta_\x$ represents the Dirac delta
function at $\x$.  The $\h$ in the subscript of $\distphat_{\h}^{N}$
is used to distinguish that this sample average distribution relies on 
observations of the samples.
Then, a distributionally robust
utility for agent $i \in \agt$ is defined as
\begin{align}
  \U[DR](s)_i(\s_i,\s_{-i})
  \ldef \inf_{\distq \in \ball_{\epsilon_i}(\distphat_{\h}^{N})}
  \mathbb{E}_{\xi \sim \distq} \big[ U_i(\s_i,\s_{-i};\xi) \big]\,,
	\label{eq:utility_dro_shared}
\end{align}
where, recall, $\ball_{\epsilon_i}(\distphat_{\h}^{N})$ is the
Wasserstein ball of radius $\epsilon_i$ centered at
$\distphat_{\h}^{N}$. 
Even though $\distphat_{\h}^{N}$ relies on global
information and is common among the agents,
$\epsilon_i \in \real_{\geq 0}$ is a parameter that  agent
$i \in \agt$ can independently choose.  The above utility description
allows us to formulate the \emph{distributionally robust game},
\begin{align}
  \game[DR](s)
  \ldef \left<\agt,\{\set_i\}_{i \in \agt},\{\U[DR](s)_i\}_{i \in \agt} \right>.
	\label{eq:game_dro_shared}
\end{align}
We refer to the set of NE of $\game[DR]$, \emph{i.e.}
$\ne(\game[DR])$, as \emph{distributionally robust Nash equilibria} or
\emph{DRoNE}.

\paragraph*{When uncertainty is individual:} A special case of the
previous scenario occurs when the random variable
$\xi \sim \bbp^\star$ affects each agent's utility in an independent
manner. We state this precisely next.
\begin{assum}\thmtitle{Individual uncertainty}
\label{asmp:ind_uncertain}
The random variable $\xi \sim \bbp^\star$ satisfies, for each agent
$i \in \agt$;
\begin{enumerate}
\item
  $\xi \in \Xi = \Xi_1 \times \cdots \times \Xi_n \subseteq \real^m$
  can be separated as $\xi = [\xi_i^\top, \cdots, \xi_n^\top]^\top$,
  with $\xi_i \in \Xi_i \subseteq \real^{m_i}$;
	
\item $\bbp^\star = \bbp^\star_1 \otimes \cdots \otimes \bbp^\star_n$,
  with $\xi_i \sim \bbp^\star_i$;
	
\item $\h_i$ is such that $\forall \xi \in \Xi$,
  $\h_i(\xi) = \xi_i \in \Xi_i \subseteq \real^{m_i}$; and
  
\item $\forall \, \s \in \set$, $U_i(\s;\xi) = U_i(\s;\xi_i)$,
  $\forall \, \xi \in \Xi$.  \bulletend
\end{enumerate}
\end{assum}
This makes it so that each agent's uncertainty does not affect another
agent in any way. When this is the case, with a slight abuse of notation,
we assume the following.
\begin{assum}\thmtitle{Individual i.i.d.~samples}
  \label{asmp:ind_sample}
  For each $i \in \agt$, $\{\xih_i^{(1)},\cdots, \xih_i^{(N_i)}\}$ are
  i.i.d. samples from $\bbp_i^\star$.  \bulletend
\end{assum}
Thus, each agent is capable of individually measuring the uncertainty
it is subject to. Moreover, we allow different $N_i \in \intnonneg$
for each agent $i \in \agt$, which means that agents can now hold
different numbers of samples from the original distribution.

In this special case, each agent can create an ambiguity set by
considering the sample average distribution
\begin{align*}
  \distphat_i^{N_i} \ldef \frac{1}{N_i} \sum_{k = 1}^{N_i} \delta_{\xih_i^{(k)}}\,.
\end{align*}
Then, the distributionally-robust utility for agent $i \in \agt$ can
be modified as 
\begin{align}
  \Udro_i(\s_i,\s_{-i}) \ldef
  \inf_{\distq \in \ball_{\epsilon_i}(\distphat_i^{N_i})}
  \mathbb{E}_{\xi \sim \distq} \big[ U_i(\s_i,\s_{-i};\xi) \big]\,,
	\label{eq:utility_dro}
\end{align}
where, $\ball_{\epsilon_i}(\distphat_i^{N_i})$ is the
Wasserstein ball of radius $\epsilon_i$ around $\bbph_i^{N_i}$. 
The above utility description allows us to reformulate the
distributionally robust game as
\begin{align}
  \game[DR] \ldef \left<\agt,\{\set_i\}_{i \in \agt},
  \{\Udro_i\}_{i \in \agt} \right>.
	\label{eq:game_dro}
\end{align}


\paragraph*{Communication protocol:} In order to update its own
strategy, an agent will need information about the strategies and
possibly other data of other agents. In this paper, we assume that the
agents obtain information by communicating with others over a static
communication network $\grph \ldef (\agt,\edg)$ with vertex set
$\agt$. The arc set $\edg$ defines the connections between agents,
with $(i,j) \in \edg$ if and only if $i \in \agt$ can send information
to $j \in \agt$. 
We make the following assumption regarding the
network.
\begin{assum}\label{asmp:network_conn}
  \thmtitle{Connectivity} The communication graph
  $\grph = (\agt, \edg)$ is strongly connected.  \bulletend
\end{assum}
This allows information from each agent to reach every other agent in
the group. Note that we fix the communication graph to be static to
make the notation simpler; however, the results in the sequel can be
extended to time-varying graphs with periodic strong connectivity.

Now, we are ready to state the goals of this work.
\begin{prob}
  \label{prob}
  Given the aforementioned setup, under the Standing
  Assumption~\ref{asmp:gen_game_assumption}, and
  Assumption~\ref{asmp:shared_sample}, on shared i.i.d.~samples,
  Assumption~\ref{asmp:obs_continuous}, on proper observations,
  and~\ref{asmp:network_conn}, on connectivity; provide
  \begin{enumerate}
  \item \label{prob:drone_exist} conditions under which a DRoNE exists;
  \item \label{prob:drone_relate} relations between DRoNE and NE
    of the stochastic game; \emph{i.e.} $\game[S]$ and $\game[DR](s)$; and
  \item \label{prob:distributed_drones} a distributed algorithm that
    converges to a DRoNE.
  \end{enumerate}
  Further, discuss these solutions under the special case given by
  Assumption~\ref{asmp:ind_uncertain}, on individual uncertainty, and
  Assumption~\ref{asmp:ind_sample}, on individual i.i.d.~samples.
  \bulletend
\end{prob}
We would like to point out that providing a solution for
Problem~\ref{prob}:~\eqref{prob:distributed_drones}, under
Assumption~\ref{asmp:shared_sample} is difficult. Due to this, we
provide a solution for Problem~\ref{prob}:~\eqref{prob:drone_exist}
and~\eqref{prob:drone_relate} in the next section, then address
Problem~\ref{prob}:~\eqref{prob:distributed_drones} for the special
case given by Assumptions~\ref{asmp:ind_uncertain}
and~\ref{asmp:ind_sample}, which we adapt later to the more general
case in Section~\ref{sec:dist_drones_shared}.

\section{On the Nash Equilibria of the Games}\label{sec:ne}

In this section, we deal with Problem~\ref{prob}:
\eqref{prob:drone_exist},~\eqref{prob:drone_relate}, and provide
qualitative properties of the Nash equilibria of the two games;
$\game[S]$ and $\game[DR](s)$.  In particular, we provide conditions
under which the sets of Nash equilibria are non-empty, and then we
relate the Nash equilibria sets.

First, we show that the distributionally robust
utility enjoys helpful properties under certain assumptions.
\begin{lem}
\label{lem:dro_utility_concave}
\thmtitle{Concavity and continuity of distributionally robust utility}
Suppose the Standing
Assumptions~\ref{asmp:gen_game_assumption}:~\eqref{asmp:strategy_compact}
and~\ref{asmp:gen_game_assumption}:~\eqref{asmp:utility_concave}
hold. Then, $\U[DR](s)_i(\cdot,\s_{-i})$ is concave
$\forall \s_{-i} \in \set_{-i}$, for each $i \in \agt$. Moreover, for
each $i \in \agt$, $\U[DR](s)_i$ is continuous.
\end{lem}
\begin{proof}
First, we show concavity. Fix an $i \in \agt$ and $\s_{-i} \in \set_{-i}$.
  From the hypothesis, the function
  $g_\distq : \real^{n_i} \to \real$,
\begin{align*}
  g_\distq(\x) \ldef \mathbb{E}_{\xi \sim \distq} \big[ U_i(\x,\s_{-i},\xi) \big]
\end{align*}
is concave. 
From the concavity of $g_\distq$ and the defining
property of infimum,
\begin{align*}
  & \U[DR](s)_i \big(\lambda\x + [1-\lambda]\y,\s_{-i}\big) \ldef \inf_{\distq \in \ball_{\epsilon_i}(\bbph_{\h}^{N})} g_\distq \big(\lambda\x + [1-\lambda]\y\big) \\
  & \qquad \geq 
    \lambda \inf_{\distq \in \Mcal{B}_{\epsilon_i}(\bbph_{\h}^{N})} g_\distq(\x) + [1-\lambda] \inf_{\distq \in \Mcal{B}_{\epsilon_i}(\bbph_{\h}^{N})} g_\distq(\y) \\
	& \qquad \rdef \lambda\U[DR](s)_i \big(\x,\s_{-i}\big) + [1-\lambda] \U[DR](s)_i \big(\y,\s_{-i}\big),
\end{align*}
for any $\x,\y \in \real^m$, and $\lambda \in [0,1]$. Hence, the first claim.

Next, we prove continuity. To begin, notice that $\set$, $\Xi$ are
compact sets and $\U_i$ is continuous on $\set \times \Xi$. Hence,
$\U_i$ is bounded on $\set \times \Xi$. Then, applying the dominated
convergence theorem~\cite{HLR-PF:88}, for any
$\bbq \in \mathcal{M}(\Xi)$, $\expect[\xi \sim \bbq]{\U_i(\s;\xi)}$ is
continuous in $\s$. Now, $\mathcal{M}(\Xi)$ is a subset of a reflexive
normed vector space. Moreover,
$\ball_{\epsilon_i}(\bbph_{\h}^{N}) \subseteq \mathcal{M}(\Xi)$ is
non-empty, bounded, closed, and convex (by construction). Then, by
Tonelli's theorem~\cite{CC:20},
$\inf_{\distq \in \ball_{\epsilon_i}(\distphat_{\h}^{N})}
\mathbb{E}_{\xi \sim \distq} \big[ U_i(\s_i,\s_{-i};\xi) \big]$
attains its minimum; \emph{i.e.} $\forall \s \in \set$,
$\exists \, \bbq_{\s} \in \mathcal{M}(\Xi)$, such that
$\U[DR](s)_i(\s) = \mathbb{E}_{\xi \sim \bbq_{\s}} \big[ U_i(\s;\xi)
\big]$. We have used the subscript to show that this $\bbq_{\s}$ can
depend on $\s$.

With this setup, we prove continuity using contradiction. Suppose
$\U[DR](s)_i$ is not continuous at $\sh \in \set$. Then,
$\exists \, a_1 > 0$, such that $\forall \, a_2 > 0$,
$\exists \, \s \in \set$ with $\|\sh-\s\| < a_2$ and
$|\U[DR](s)_i(\sh) - \U[DR](s)_i(\s)| > a_1$. Assume w.l.o.g. that
$\U[DR](s)_i(\sh) > \U[DR](s)_i(\s) + a_1$. Then, it is easy to see
from the definition in~\eqref{eq:utility_dro} 
and the prior discussion,
that this would imply that
$\exists \, \bbq_{\s}, \bbq_{\hat{\s}} \in \ball_{\epsilon_i}(\distphat_{\h}^{N})$, such that
$\expect[\xi \sim \bbq_{\hat{\s}}]{\U_i(\sh;\xi)} > \expect[\xi \sim
\bbq_{\s}]{\U_i(\s;\xi)} - a_1$. Now as
$\bbq_{\sh} \in \argmin_{\bbq \in \ball_{\epsilon_i}(\distphat_{\h}^{N})}\mathbb{E}_{\xi \sim
  \bbq} \big[ U_i(\sh;\xi) \big]$,
$\expect[\xi \sim \bbq_{\s}]{\U_i(\sh;\xi)} \geq \expect[\xi \sim
\bbq_{\sh}]{\U_i(\sh;\xi)} > \expect[\xi \sim \bbq_{\s}]{\U_i(\s;\xi)}
- a_1$.  This then contradicts the continuity of
$\expect[\xi \sim \bbq_{\s}]{\U_i(\cdot;\xi)}$. 
\end{proof}

Using this result, we can immediately show that the distributionally
robust game possesses a Nash equilibrium.

\begin{prop}
\label{prop:drone_existence}
\thmtitle{Existence of a DRoNE} Suppose that the Standing
Assumption~\ref{asmp:gen_game_assumption} holds. Then,
$\ne(\game[DR](s)) \neq \emptyset$. Further,
$\ne(\game[S]) \neq \emptyset$.
\end{prop}

\begin{proof}
  We provide the proof for $\ne(\game[DR](s)) \neq \emptyset$, as the
  proof of $\ne(\game[S]) \neq \emptyset$ is analogous.  Consider an
  arbitrary but fixed $i \in \agt$.  From
  Lemma~\ref{lem:dro_utility_concave}, $\U[DR](s)_i(\cdot,\s_{-i})$ is
  concave $\forall s_{-i}\in \set_{-i}$ and $\U[DR](s)_i$ is continuous. 
  Moreover, $\set_i$ is a compact set, $\forall i \in \agt$, which
  makes $\set = \times_{i\in \agt} \set_i$ also compact.  Thus,
  $\U[DR](s)_i$ is continuous, defined on a compact set, and concave
  in its first argument $\forall i \in \agt$. The existence claim
  follows from the application of Kakutani's fixed point
  theorem~\cite{YN:14}.
\end{proof}

Note that the Standing
Assumptions~\ref{asmp:gen_game_assumption}:~\eqref{asmp:strategy_compact}
and~\eqref{asmp:utility_concave} are
standard in showing the existence of Nash equilibria in
games~\cite{YN:14}. As such, Proposition~\ref{prop:drone_existence}
extends this to distributionally robust games in a seamless way.


Now that we have shown that the set of DRoNE's is non-empty, we are
interested in relating the DRoNE set with the Nash equilibria of the
stochastic game. In order to do this, we first need to extend the
result in Theorem~\ref{thm:dro_bound} to account for samples based on
observations. For the sake of brevity, we define
\begin{align}
  \distphat_{\h, \Sigma}^{N} \ldef \frac{1}{N} \sum_{k = 1}^{N}
  \distgauss{\h(\hat{\xi}^{(k)})}{\Sigma}\,,
\end{align} 
as the density that smoothens out sample point masses using a Gaussian
of covariance $\Sigma$. Similarly, denote
$\distphat_{\Sigma}^{N} \ldef \frac{1}{N} \sum_{k = 1}^{N}
\distgauss{\xih^{(k)}}{\Sigma}$. We restrict our observation functions
$\h$ to be in the following class.

\begin{defn}\label{def:obs_infer}
  \thmtitle{Inferable observations} Consider any matrix sequence
  $\{\Sigma_l\}_{l\in\intpos}$ that satisfies $\Sigma_l \to \bld{O}$,
  as $l \to \infty$.  For each $l \in \intpos$ define the set of
  mappings
  $\Tset_l \ldef \{T \,|\, T_\sharp\distphat_{\Sigma_l}^N =
  \distp^\star\}$. Suppose $\h$ is such that $\forall T \in \Tset_l$,
\begin{align*}
	\int_{\real^m }\|\xi - \h(\xi)\|_1  \bbph_{\Sigma_l}^{N}(\dg \xi) \leq C_l \int_{\real^m }\|\xi - T(\xi)\|_1  \bbph_{\Sigma_l}^{N}(\dg \xi)\,.
\end{align*}
We call $\h$ \emph{inferable} with \emph{inflation} $C$, if
$\lim\limits_{l \to \infty} C_l = C$.  \bulletend
\end{defn}
The property of Definition~\ref{def:obs_infer} allows us to inflate
the Wasserstein ball in Theorem~\ref{thm:dro_bound}, and account for
the indirect observation of the samples. This leads to the following
result. 


\begin{thm}
\label{thm:dro_bound_extend}
\thmtitle{Distributionally robust bound for observations} 
Suppose Assumptions~\ref{asmp:shared_sample}, and~\ref{asmp:obs_continuous} hold; and
let $\h$ be inferable with inflation $C$.
Fix $\theta \in (0,1)$. Then,
$\Pr \big\{ \bbp^\star \in \mathcal{B}_{[C+1]\epsilon} (\distphat_{\h}^{N})
\big\} \geq 1-\theta$, where $\epsilon$ is as in~\eqref{eq:dro_bound}.
\end{thm} 

\begin{proof}
  We use the Monge form of the optimal transport problem
  in~\eqref{eq:monge_dist} to prove this claim.  Recall that
  $\distphat^{N} = \frac{1}{N} \sum_{k = 1}^{N} \delta_{\xih^{(k)}}$
  and
  $\distphat_{\Sigma}^{N} = \frac{1}{N} \sum_{k = 1}^{N}
  \distgauss{\xih^{(k)}}{\Sigma}$. Suppose $\Fcol$ is the Borel $\sigma$-algebra on $\real^m$. Then, notice that for the
  smoothened density,
\begin{align*}
  \h_{\sharp}\bbph_{\Sigma}^{N} & = \h_{\sharp}
                                  \bigg[ \frac{1}{N} \sum_{k = 1}^{N} \distgauss{\xih^{(k)}}{\Sigma}\bigg]
                                  = \frac{1}{N} \sum_{k = 1}^{N} \h_{\sharp}\distgauss{\xih^{(k)}}{\Sigma} = \bbph_{\h,\Sigma}^{N}.
\end{align*}
Take any sequence $\{\Sigma_l\}_{l\in\intpos}$ as in
Definition~\ref{def:obs_infer}. Then, for any $l \in \intpos$, and any
set $\Fset \in \Fcol$ 
we have that
\begin{align*}
  \h_{\sharp}\bbph_{\Sigma_l}^{N}(\Fset)
  \!=\! \frac{1}{N} \sum_{k = 1}^{N} \h_{\sharp}\distgauss{\xih^{(k)}}{\Sigma_l}(\Fset)
  \!=\! \frac{1}{N} \sum_{k = 1}^{N} \Pr\big\{\h(X_k) \in \Fset \big\},
\end{align*}
with $X_k \sim \distgauss{\xih^{(k)}}{\Sigma_l}$,
$\forall \, k \in \{1,\cdots,N\}$. Thus,
${\h_{\sharp}\bbph_{\Sigma_l}^{N} = \bbph_{\h, \Sigma_l}^{N}}$.
  
Now, consider a particular $l \in \intpos$. Recall that
$\Tset_l \ldef \{T \,|\, T_\sharp\distphat_{\Sigma_l}^N =
\distp^\star\}$ and define the set of mappings
$\tilde{\Tset}_l \ldef \{\tilde{T} \,|\,
\tilde{T}_\sharp\distphat_{\h, \Sigma_l}^N = \distp^\star\}$. Since
$\h$ is invertible, for every $T$ such that
$T_\sharp\distphat_{\Sigma_l}^N = \distp^\star$ (\emph{i.e.}
$T \in \Tset_l$) it holds that
$\tilde{T}_\sharp\distphat_{\Sigma_l}^N =
T_\sharp(\h^{-1})_\sharp\distphat_{\h, \Sigma_l}^N =
T_\sharp\distphat_{\Sigma_l}^N = \distp^\star$, for
$\tilde{T} = T \circ (\h^{-1})$ (\emph{i.e.}
$\tilde{T} \in \tilde{\Tset}_l$). Similarly,
$\forall \tilde{T} \in \tilde{\Tset}_l$, we have $T \in \Tset_l$.

Next, define the following operators
$F : \Tset_l \to \tilde{\Tset}_l$, with $F(T) = T \circ (\h^{-1})$ and
$G : \tilde{\Tset}_l \to \Tset_l$, with
$G(\tilde{T}) = \tilde{T} \circ \h$.  It is easy to see that $F$ and
$G$ are inverse operators. 
With this,
\begin{align*}
  & \dist[W](\bbph_{\h, \Sigma_l}^{N}, \distp^\star) =
    \inf_{\tilde{T} \in \tilde{\Tset}_l} \int_{\real^m }\|\xi - \tilde{T}(\xi)\|_1 \, \bbph_{\h, \Sigma_l}^{N}(\dg \xi) \\
  & \qquad \quad = \inf_{T \in \Tset_l} \int_{\real^m }\|\xi - \big(T \circ (\h^{-1})\big)(\xi)\|_1 \, \bbph_{\h, \Sigma_l}^{N}(\dg \xi).
\end{align*}
Since
$\h_\sharp\bbph_{\Sigma_l}^{N}(\dg \xi) = \bbph_{\h, \Sigma_l}^{N}(\dg
\xi)$, using the change of variables for pushforward measures, we have,
\begin{align*}
	& \int_{\real^m }\|\xi - \big(T \circ (\h^{-1})\big)(\xi)\|_1 \, \bbph_{\h, \Sigma_l}^{N}(\dg \xi) \\
	& = \int_{\real^m }\|\h(\xi) - \big(T \circ (\h^{-1})\big)(\h(\xi))\|_1  \bbph_{\Sigma_l}^{N}(\dg \xi) \\
	& = \int_{\real^m }\|\h(\xi) - T(\xi)\|_1  \bbph_{\Sigma_l}^{N}(\dg \xi) \\
	& \leq \int_{\real^m }\|\h(\xi) - \xi\|_1  \bbph_{\Sigma_l}^{N}(\dg \xi) + \int_{\real^m } \!\!\|\xi - T(\xi)\|_1  \bbph_{\Sigma_l}^{N}(\dg \xi) \,.
\end{align*}
The last inequality comes from triangle inequality and linearity of integrals. Then, because of Definition~\ref{def:obs_infer},
\begin{align*}
	& \dist[W](\bbph_{\h, \Sigma_l}^{N}, \distp^\star) = \inf_{T \in \Tset} \! \int_{\real^m } \!\!\!\!\!\!\|\xi - \big(T \circ (\h^{-1})\big)(\xi)\|_1 \, \bbph_{\h, \Sigma_l}^{N}(\dg \xi)\\
	& \leq \inf_{T \in \Tset} \left[\int_{\real^m } \!\!\!\!\!\! \|\h(\xi) - \xi\|_1  \bbph_{\Sigma_l}^{N}(\dg \xi) + \int_{\real^m } \!\!\!\!\!\!\!\|\xi - T(\xi)\|_1  \bbph_{\Sigma_l}^{N}(\dg \xi) \right] \\
	& \leq [C_l + 1] \inf_{T \in \Tset} \int_{\real^m }\|\xi - T(\xi)\|_1  \bbph_{\Sigma_l}^{N}(\dg \xi) \\
	& = [C_l+1] \,\dist[W](\bbph_{\Sigma_l}^{N}, \distp^\star)\,.
\end{align*}
Now, as $l \to \infty$, $\dist[W](\bbph_{\h, \Sigma_l}^{N}, \distp^\star) \to \dist[W](\bbph_{\h}^{N}, \distp^\star)$ and $\dist[W](\bbph_{\Sigma_l}^{N}, \distp^\star) \to \dist[W](\bbph^{N}, \distp^\star)$. Thus, since $C_l \to C$ as $l \to \infty$, taking the limit as $l \to \infty$ of the previous inequality, we get
\begin{align*}
	\dist[W](\bbph_{\h}^{N}, \distp^\star) \leq [C+1] \, \dist[W](\bbph^{N}, \distp^\star)\,.
\end{align*} 
Then, by applying Theorem~\ref{thm:dro_bound}, the proof is complete.
\end{proof}

\begin{rem}\thmtitle{On sufficient conditions for
    Theorem~\ref{thm:dro_bound_extend}}
  \label{rem:suff_cond_strictness} 
  \rm The uncertainty bound in Theorem~\ref{thm:dro_bound_extend}
  relies on the ``inflation'' of the original Wasserstein ball in
  order to guarantee that the sample average using the observations
  lies in this larger ball with similar probability. The existence of
  this inflation, $C+1$, depends on the observation function $\h$
  through Assumption~\ref{asmp:obs_continuous} and
  Definition~\ref{def:obs_infer}. Notice, in fact, that $\h$ can
  itself be viewed as a transport map. Thus, the assumption roughly
  captures that, when $\h$ is inferable with inflation $C$, the effect
  of the transport defined by $\h$ is on average proportional to the
  effect of transport under the optimal transport map. 
    \bulletend
\end{rem}

Thus, using the uncertainty quantification provided by observations we
provide the following relation between the sets of Nash equilibria.

\begin{thm}
\label{thm:drone_epsilon_ne_shared}
\thmtitle{Every DRoNE is an $\eta$-NE of $\game[S]$ with certain
  probability} Suppose Assumption~\ref{asmp:shared_sample} holds and
that $\h$ is inferable with inflation $C$.  Suppose that, for all
$i \in \agt$, $\theta_i \in (0,1)$ and $\epsilon_i(N,\theta_i)$ are
chosen according to~\eqref{eq:dro_bound}.  Assume that the strategy
$(\sst_i,\sst_{-i}) \in \ne(\game[DR](s))$, where the $\inf$
in~\eqref{eq:utility_dro_shared} is taken over closed Wasserstein
balls of radii $\bar{\epsilon}_i = [C+1] \epsilon_i$, \emph{i.e.}
$\bbq \in \ball_{\bar{\epsilon}_i}(\distphat_\h^{N})$,
$\forall i \in \agt$. Then, with probability at least $(1-\theta)^n$,
$(\sst_i,\sst_{-i}) \in \ne_\eta(\game[S])$,
where 
$\eta = 2 \,[C+1] \max_{i \in \agt} \epsilon_i L_i$, and
$\theta = \max_{i \in \agt} \theta_i$.
\end{thm}

\begin{proof}
%
  Consider an arbitrary but fixed $i \in \agt$. Since
  $\bar{\epsilon}_i$ is chosen according to~\eqref{eq:dro_bound}, then
  by Theorem~\ref{thm:dro_bound_extend},
  $\distp^\star \in \ball_{\bar{\epsilon}_i}(\distphat_\h^{N})$ with
  probability at least $1-\theta_i$.  Thus,
\begin{align}
	\U[DR](s)_i(\sst_i,\sst_{-i}) \leq \U[S]_i(\sst_i,\sst_{-i})\,,
	\label{eq:udr_leq_us_ne}
\end{align} 
with probability at least $1-\theta_i$.  Now, from the hypothesis we
have that $(\sst_i,\sst_{-i}) \in \ne(\game[DR](s))$. This gives us,
\begin{align}
	\U[DR](s)_i(\sst_i,\sst_{-i}) \geq \U[DR](s)_i(\s_i,\sst_{-i}), \forall \s_i \in \set_i\,.
	\label{eq:udr_ne_inequality}
\end{align}
Next, consider an arbitrary but fixed $\s_i \in \set_i$ and let
$\bbq \in
\ball_{\bar{\epsilon}_i}(\distphat_i^{N_i})$. 
Thus, again from Theorem~\ref{thm:dro_bound_extend},
$\dist[W](\bbp_i^\star,\bbq) \leq 2\,
\bar{\epsilon}_i = 2\,C\,\epsilon_i$, with probability at least
$1-\theta_i$. Thus, using the Kantorovich–Rubinstein description of
the Wasserstein distance in~\eqref{eq:kr_dist}, we have, 
\begin{align}
  \notag & [1/L_i] \Big[ \U[S]_i(\s_i,\sst_{-i}) - \expect[\xi \sim \bbq]{\U_i(\s_i,\sst_{-i};\xi)} \Big] \\
  \notag &
           = \frac{1}{L_i}
           \bigg[ \int_{\Xi} \U_i(\s_i,\sst_{-i};\xi) \,\bbp_i^\star(\dg \xi) - \int_{\Xi} \U_i(\s_i,\sst_{-i};\xi) \,\bbq(\dg \xi)\bigg] \\
	\label{eq:utility_dro_s_diff_temp} &  \qquad \leq 2\, [C+1] \,\epsilon_i\,.
\end{align}
with probability at least $1-\theta_i$. Here, the first equality comes
from the definitions of the utilities and the last inequality comes
from~\eqref{eq:kr_dist}, Standing
Assumption~\ref{asmp:gen_game_assumption}:~\eqref{asmp:utility_xi_lipschitz}
and the prior discussion. Thus,
using~\eqref{eq:utility_dro_s_diff_temp},
\begin{align}
	\notag & \U[S]_i(\s_i,\sst_{-i}) - \U[DR](s)_i(\s_i,\sst_{-i})\\
	\notag & =  \U[S]_i(\s_i,\sst_{-i}) - \inf_{\bbq \in \ball_{\bar{\epsilon}_i}(\distphat_i^{N_i})} \expect[\xi \sim \bbq]{\U_i(\s_i,\sst_{-i};\xi)} \\
	\notag & = \sup_{\bbq \in \ball_{\bar{\epsilon}_i}(\distphat_i^{N_i})} \Big[ \U[S]_i(\s_i,\sst_{-i}) - \expect[\xi \sim \bbq]{\U_i(\s_i,\sst_{-i};\xi)} \Big] \\
	\label{eq:utility_dro_s_diff} & \qquad \leq 2\, [C+1] \,\epsilon_i \, L_i\,.
\end{align}
Moreover, notice that the upper bound in independent of the chosen
$\s_i$.  Hence, by
combining~\eqref{eq:udr_leq_us_ne},~\eqref{eq:udr_ne_inequality},
and~\eqref{eq:utility_dro_s_diff},
\begin{align*}
	\forall s_i \in \set_i, & \quad  \U[S]_i(\sst_i,\sst_{-i}) \geq \U[DR](s)_i(\sst_i,\sst_{-i}) \geq \U[DR](s)_i(\s_i,\sst_{-i}) \\
	&\!\! \geq \U[S]_i(\s_i,\sst_{-i}) - 2 [C+1] \epsilon_i L_i \geq \U[S]_i(\s_i,\sst_{-i}) - \eta,
\end{align*}
with probability at least $1-\theta_i \geq 1-\theta$. Here, $\eta$ and
$\theta$ are as defined in the hypothesis. Rearranging the last
equation and further applying it to all agents completes the proof.
\end{proof}

\begin{rem}
\label{rem:local_lipschitz_constant}
\thmtitle{On relaxing the Lipschitz dependence of utility on
  uncertainty} \rm In
Assumption~\ref{asmp:gen_game_assumption}:~\eqref{asmp:utility_xi_lipschitz},
we require $\U_i(\s;\xi)$ to be $L_i$-Lipschitz in $\xi$,
$\forall \, \s \in \set$. It is possible to relax this condition to
require a Lipschitz constant $l_i(\s)$ that depends on the strategy
profile $\s \in \set$; and obtain the exact same result as in
Theorem~\ref{thm:drone_epsilon_ne_shared} with bound
$\eta = 2 [C+1] \max_{i \in \agt, \s \in \set} \epsilon_i l_i(s)$. The extra
requirement would impose that $\max_{\s \in \set} l_i(s)$ exists
$\forall \, i \in \agt$; which, in turn, is equivalent to
Assumption~\ref{asmp:gen_game_assumption}:~\eqref{asmp:utility_xi_lipschitz}
to begin with.  \bulletend
\end{rem}

The bounds provided in the previous result take into account the
heterogeneity among the agents in terms of the dependence (through
$L_i$) of their utility on the random variable and the size of their
ambiguity sets (due to $\epsilon_i$). Informally, for each agent
$i \in \agt$, $\epsilon_i$ and $\theta_i$ can be be made smaller as
the number of samples $N_i$ grows (see~\cite{PME-DK:17}). Hence, the
DRoNE becomes a NE of the stochastic game with hight probability as
the agents gather more samples from the unknown distribution. In this
paper, we skip formal details regarding this matter, since we keep
$N_i$ fixed $\forall \, i \in \agt$.

The previous result can be understood as follows. Suppose that every
agent $i \in \agt$ chooses $\epsilon_i$ independently, then
Theorem~\ref{thm:drone_epsilon_ne_shared} quantifies the uncertainty
that the DRoNE belongs to $\ne(\game[S])$ using $\theta_i$
accordingly.
We conclude this section
by adapting the previous result for the special case with
Assumptions~\ref{asmp:ind_uncertain}, and~\ref{asmp:ind_sample}. Here,
we can use the $\epsilon, \theta$ relation in
Theorem~\ref{thm:dro_bound} directly for the $\eta$ bound, since there
is no effect of observations. We skip the proof since it follows the
exact same arguments as Theorem~\ref{thm:drone_epsilon_ne_shared}.

\begin{cor}
\label{cor:drone_epsilon_ne}
\thmtitle{Under individual uncertainty, every DRoNE is an $\eta$-NE of
  $\game[S]$ with certain probability} Suppose
Assumptions~\ref{asmp:ind_uncertain}, and~\ref{asmp:ind_sample}
hold. Suppose $\forall i \in \agt$, $\theta_i \in (0,1)$ and
$\epsilon_i(N_i,\theta_i)$ is chosen according
to~\eqref{eq:dro_bound}.  Assume that the strategy
$(\sst_i,\sst_{-i}) \in \ne(\game[DR])$ with Wasserstein ball radii
$\epsilon_i$. Then with probability at least $(1-\theta)^n$,
$(\sst_i,\sst_{-i}) \in \ne_\eta(\game[S])$, with
$\eta = 2 \max_{i \in \agt} \epsilon_i L_i$ and
$\theta = \max_{i \in \agt} \theta_i$.  
\proofend
\end{cor}


\section{A Centralized Better Response Supergradient Ascent
  Dynamics} \label{sec:algo}
In the previous section, we established the existence of a DRoNE and
showed its relation to the NE of the stochastic game. 
In this section, we provide a centralized algorithm that allows the
agents to learn said DRoNE. 
We build up to a distributed solution that addresses
Problem~\ref{prob}:~\eqref{prob:distributed_drones} in the next
section.

First, recall that Lemma~\ref{lem:dro_utility_concave} shows that for
each agent $i \in \agt$, $\U[DR](s)_i$ is concave in its own strategy
if the opponents' strategies are fixed. This means that (due to
Definition~\ref{def:subgr}) there is a non-empty set of supergradients
of $i$'s utility with respect to its own strategy (for fixed
opponents' strategies).  We assume an additional feature of the
distributionally robust utilities as follows.
\begin{assum}\label{asmp:utility_bounded_gradient}
  \thmtitle{Uniformly bounded supergradient} For each $i \in \agt$,
  $\exists \, B_i \in \real_{\geq 0}$ such that $\forall \s \in \set$,
  $\|\zetab_i\| \leq B_i$,
  $\forall \zetab_i \in \del_\x [\U[DR](s)_i(\x,\s_{-i})]
  |_{\s_i}$. \bulletend
\end{assum}

From the definitions of the NE in Definition~\ref{def:eps_ne} and the
supergradient in Definition~\ref{def:subgr}, it is well known
that an alternative characterization of the Nash equilibria comes in
the form of the following variational inequality~\cite{RTR:18},
\begin{align*}
  \notag \st \in \ne(\game[DR](s)) & \hspace{-1ex}\iff\hspace{-1ex}
                                     \forall \, \s \in \set, \exists \,
                                     \zetab \in \bigtimes_{i \in \agt} \del_\x  \Big[ \U[DR](s)_i(\x,\s_{-i}) \Big] \Big|_{\s_i} \\
                                   & \qquad \text{such that} \,\,\,  (\s-\st)^\top\zetab \leq 0 \,.
\end{align*}
Hence, inspired by~\cite{YWC-CK-MA:24}, we define the following
Lyapunov function candidate, used for analysis later.
\begin{subequations}
\begin{align}
  V(\s,\phib) & \ldef \max_{\x \in \set} \x^\top \phib - \s^\top \phib \\
  \label{eq:lyap_separate}
              & = \sum_{i \in \agt} \Big[ \max_{\x_i \in \set_i } \x_i^\top \phib_i - \s_i^\top \phib_i \Big] \,. 
\end{align}
\label{eq:lyap}
\end{subequations}
Here, $\phib_i$ (like $\s_i$) are the components of a variable $\phib$
(to be defined later) corresponding to agent $i \in \agt$. The equality
in~\eqref{eq:lyap_separate} comes from the fact that
$\max\limits_{\x \in \set} \x^\top \phib$
$= \max\limits_{\x \in \set} \sum\limits_{i \in \agt} \x_i^\top
\phib_i$
$= \sum\limits_{i \in \agt} \max\limits_{\x_i \in \set_i} \x_i^\top
\phib_i$. This is because $\set = \times_{i \in \agt} \set_i$ and the
decision variables do not affect each other through other
constraints.
Note that $V$ is Lipschitz continuous in its arguments as
$\max_{\x \in \set} \x^\top \phib$ is convex in $\phib$, $\set$ is
compact (hence $\max_{\x \in \set} \x^\top \phib$ has uniformly
bounded subgradients), and $\s^\top\phib$ is bilinear. Moreover, by
construction, $V(\s,\phib) \geq 0$, $\forall \, \s, \phib$. The next
result (whose proof is in Appendix~\ref{app:proofs}) bounds the change
in $\max_{\x \in \set} \x^\top \phib$ w.r.t.~$\phib$.
\begin{lem}\thmtitle{On the change of max of linear functions}
\label{lem:change_max_lin_func}
For an agent $i \in \agt$, define
$f_i(\phib_i) \ldef \max\limits_{\x_i \in \set_i} \x_i^\top \phib_i$,
where $\set_i$ is a compact set with diameter $D_i \in \realpos$, and
let
\begin{align}
  \Xset^*_i(\phib_i) \ldef \argmax_{\x_i \in \set_i} \x_i^\top \phib_i\,.
  \label{eq:argmax_x_star}
\end{align}
Then, $\forall \phib_i^1, \phib_i^2 \in \ball_{M_i}(\zero)$, and
$\forall \x_i^* \in \Xset^*_i(\phib_i^1)$,
\begin{align*}
  f_i(\phib_i^2) \leq  &
                         \,\, f_i(\phib_i^1)
                         + \ip{\x_i^*}{[\phib_i^2 - \phib_i^1]}
                         + \bar{D}_i \|\phib_i^2 - \phib_i^1\|, 
\end{align*}
for any $\bar{D}_i \geq D_i$.
\bulletend
\end{lem}  

We are now ready to propose the first part of our algorithm. We allow
agents to update their strategies at every time instant in a better
response fashion. To that effect, with the set defined
  in~\eqref{eq:argmax_x_star}, suppose
\begin{align}
  \suppf_i(\phib_i) \in \Xset^*_i(\phib_i) \,.
  \label{eq:supp_func}
\end{align}
Notice here that $\suppf_i$ is related to the so-called \emph{support
  function} of the convex set $\set_i$, which motivates the name of
our algorithm.
Now, given $\{\phib(t)\}$, and using~\eqref{eq:supp_func}, agents
update their strategies via
\begin{align}
  \label{eq:dyn_s} \s_i(t+1) &
                               = [1-\alpha_i] \,
                               \s_i(t) + \alpha_i \, \suppf_i(\phib_i(t)), \\ 
  \notag & = \s_i(t)
           + \alpha_i \, \big[\suppf_i(\phib_i(t)) - \s_i(t) \big], \,\,
           \forall i \in \agt\,,
\end{align}
from some initial $\{\s_i(0) \in \set_i\}_{i \in \agt}$.  In other
words, the update of $\s_i(t)$ is a convex combination of $\s_i(t)$
and $\suppf_i(\phib_i(t))$ using the parameter $\alpha_i \in (0,1)$,
for each $i \in \agt$. Thus, $\suppf_i(\phib_i(t))$ 
renders $\set_i$ invariant
under~\eqref{eq:dyn_s}, $\forall i \in \agt$. 
Before specifying $\phib_i(t)$, we bound the change of the Lyapunov
function for any choice of bounded $\phib(t)$.
\begin{lem}\thmtitle{Bound on Lyapunov function difference}
\label{lem:lyap_change}
Consider the dynamics~\eqref{eq:dyn_s} from an initial condition
$\{\s_i(0) \in \set_i\}_{i \in \agt}$. Suppose
$\{\phib(t)\}_{t \in \intnonneg}$ satisfies $\|\phib_i(t)\| \leq M_i$
(for some $M_i \in \real_{>0}$), $\forall i \in \agt$,
$\forall t \in \intnonneg$. Take any $\bar{D}_i \geq D_i$,
$\forall i \in \agt$.
Then, $\forall t \in \intnonneg$, with
$\underline{\alpha} = \min_{i \in \agt} \alpha_i \in (0,1)$, it holds
that
\begin{align}
  \notag & V\big(\s(t+1),\phib(t+1)\big) - V\big(\s(t),\phib(t)\big) \leq \\
  \notag & \quad \sum_{i \in \agt} \Big[ \frac{1-\alpha_i}{\alpha_i}[\s_i(t+1) -
           \s_i(t)]^\top [\phib_i(t+1) - \phib_i(t)] \\ 
  \label{eq:lyap_change} & \quad + \bar{D}_i \|\phib_i(t+1) - \phib_i(t)\| \Big]
                           - \underline{\alpha} V(\s(t),\phib(t)) \,. 
\end{align}
\end{lem}   
\begin{proof}
  With a slight abuse of notation, we denote
  $\s = \s(t)$, $\sp = \s(t+1)$, $\phib = \phib(t)$, and
  $\phibp = \phib(t+1)$.
  From the definition of $V$ in~\eqref{eq:lyap} and the dynamics
  in~\eqref{eq:dyn_s}, we have that 
\begin{align*}
	& V(\sp,\phibp) - V(\s,\phib) \\
	& = \sum_{i \in \agt} \Big[ \max_{\x_i \in \set_i } \x_i^\top \phibp_i - \max_{\x_i \in \set_i } \x_i^\top \phib_i - {\sp_i}^\top \phibp_i + \s_i^\top \phib_i \Big] \\
	& = \sum_{i \in \agt} \Big[ \max_{\x_i \in \set_i } \x_i^\top \phibp_i - \max_{\x_i \in \set_i } \x_i^\top \phib_i \\
	& - \s_i^\top \big[ \phibp_i - \phib_i \big] - \big[\sp_i - \s_i \big]^\top \big[ \phibp_i - \phib_i \big] - \big[\sp_i - \s_i \big]^\top\phib_i \Big]\,.
\end{align*}
Then, we can upper bound the time difference in $V$ as
\begin{align*}
	& V(\sp,\phibp) - V(\s,\phib) \\
	& \leq \sum_{i \in \agt} \Big[ \ip{\suppf_i(\phib_i)}{[\phibp_i - \phib_i]} + \bar{D}_i \|\phibp_i - \phib_i\| \\
	& \qquad - \s_i^\top \big[ \phibp_i - \phib_i \big] - \alpha_i \big[\suppf_i(\phib_i) - \s_i \big]^\top \big[ \phibp_i - \phib_i \big] \\ 
	& \qquad - \alpha_i \big[\suppf_i(\phib_i) - \s_i \big]^\top\phib_i \Big] \,.
\end{align*}
The first two terms on the right-hand side of the previous inequality
come
from applying Lemma~\ref{lem:change_max_lin_func}; and the last two
terms similarly come from replacing~\eqref{eq:dyn_s} and rearranging
terms. Then, combining like terms gives,
\begin{align*}
	& V(\sp,\phibp) - V(\s,\phib) \\
	& \leq \sum_{i \in \agt} \Big[ [\suppf_i(\phib_i) - \s_i]^\top [\phibp_i - \phib_i] + \bar{D}_i \|\phibp_i - \phib_i\| \\ 
	& \quad - \alpha_i \big[\suppf_i(\phib_i) - \s_i \big]^\top \big[ \phibp_i - \phib_i \big] - \alpha_i \big[\suppf_i(\phib_i) - \s_i \big]^\top\phib_i \\
	& \leq \sum_{i \in \agt} \Big[ [1-\alpha_i][\suppf_i(\phib_i) - \s_i]^\top [\phibp_i - \phib_i] \\
	& \hspace{8em} + \bar{D}_i \|\phibp_i - \phib_i\| \Big] - \underline{\alpha} V(\s,\phib)\,.
\end{align*}
For the last inequality, we have used the fact that
$\sum_{i \in \agt}\big[\suppf_i(\phib_i) - \s_i \big]^\top\phib_i =
\sum_{i \in \agt} \big[ \max_{\x_i \in \set_i } \x_i^\top \phib_i -
\s_i^\top \phib_i \big] = V(\s,\phib)$, and combined terms from the
previous step. Finally, applying~\eqref{eq:dyn_s} in the last
inequality completes the proof.
\end{proof}

Recall that supergradients are possibly non-unique. Thus, an alternate
definition of the NE~\cite{RTR:18} is given by
\begin{align*}
	\st \in \ne(\game[DR](s)) \hspace{-1ex}\iff\hspace{-1ex} \zero \in \nc_{\set_i}(\sst_i) - \del_\x  \Big[ \U[DR](s)_i(\x,\sst_{-i}) \Big] \Big|_{\sst_i} .
\end{align*}
This motivates us to introduce the following min norm
supergradient. First, let $\Pi_\x(\x,\y)$ be the projection of
$(\x,\y)$ onto $\x$, the first components. Then, define
$\forall i \in \agt$,
\begin{align}\label{eq:min_supergr}
  \v_i(\s) \ldef & \,\Pi_{\zetab_i} \Big( \argmin_{\zetab_i, \z} \| -\zetab_i + \z \|^2 \\
  \notag & \,\, \mathrm{s.t.} \,\, \zetab_i \in \del_\x
           \Big[ \U[DR](s)_i(\x,\s_{-i}) \Big] \Big|_{\s_i}, \, \z \in \nc_{\set_i}(\s_i) \Big)\,.
\end{align} 
Note that the previous formulation in~\eqref{eq:min_supergr} is well
defined since $\|\cdot\|^2$ is strongly convex and lower bounded by
$0$; the set of supergradients is convex and compact; and the normal
cone is a closed convex cone. Hence,~\eqref{eq:min_supergr} produces a
unique optimizer. Further, we introduce a proximal term
$\forall i \in \agt$, with $\lambda_i \in \real_{>0}$,
\begin{align}
  \w_i(\s_i,t) \ldef - \frac{1}{\lambda_i} [\s_i - \s_i(t-1)]\,,
  \label{eq:prox_derivative}
\end{align}  
which is the derivative of the $(1/\lambda_i)$-strongly concave
function $- (1/\lambda_i) \|\s_i - \s_i(t-1)\|^2$. This penalizes the
deviation of the strategy at time $t \in \intnonneg$ from the previous
strategy, and this penalty depends on the 
multiplier $\lambda_i^{-1}$. Now, using the definitions
in~\eqref{eq:min_supergr}, and~\eqref{eq:prox_derivative}, we propose
the following choice of $\phib_i(t)$ for~\eqref{eq:dyn_s}, with
$\mu_i > 0$,
\begin{align}
  \phib_i(t) = \mu_i\v_i(\s(t)) + \w_i(\s_i(t),t)\,.
  \label{eq:phi}
\end{align}
From here, and under Assumption~\ref{asmp:utility_bounded_gradient},
it is easy to see that $\forall i \in \agt$,
$\|\phib_i(t)\| \leq \mu_i B_i + D_i/\lambda_i$,
$\forall t \in \intnonneg$.  We formally define our algorithm dynamics
next and discuss the particular choice of $\phib(t)$ in the remark
that follows.
\begin{defn}\thmtitle{\algoname} 
\label{def:isbrag}
The dynamics obtained from the strategy update in~\eqref{eq:dyn_s},
using the modified sypergradient vector in~\eqref{eq:phi} is referred to as 
\emph{Inertial Supported Better Response Ascending superGradient}
dynamics or \emph{\algoname}.  \bulletend
\end{defn}

\begin{rem}\thmtitle{On the choice of $\phib_i(t)$}
\label{rem:phi_choice}
\rm From the definitions in~\eqref{eq:min_supergr},
and~\eqref{eq:prox_derivative}, it is easy to realize that
$\forall i \in \agt$,
\begin{align*}
  & \phib_i(t) = \del_\x
    \Big[ \mu_i \U[DR](s)_i(\x,\s_{-i}(t)) \Big] \Big|_{\s_i(t)} \hspace{-0.5em}-
    \frac{1}{\lambda_i} [\s_i(t) - \s_i(t-1)] \\
  & \quad \,\, \in \del_\x  \bigg[ \mu_i
    \U[DR](s)_i(\x,\s_{-i}(t)) - \frac{1}{2\lambda_i} \|\x - \s_i(t-1)\|^2 \bigg] \bigg|_{\s_i(t)}. 	
\end{align*}
This produces the effect of allowing the agents to update their
strategies at every time instant using the knowledge of the
supergradient direction (scaled by $\mu_i$) of their utility, along
with an inertial direction so that they do not deviate from their
previous strategies too much.
\bulletend
\end{rem}

Notice from~\eqref{eq:phi} that \algoname uses a one time-step delayed
state information to perform the current state update
in~\eqref{eq:dyn_s}. In such a case, it is useful to define an
auxiliary state $\p_i(t) = \s_i(t-1)$, $\forall i \in \agt$,
$\forall t \in \intnonneg$. Then, the equilibrium set for \algoname is
given by
\begin{align}
  \!\!\!\eqpt \!\ldef\! \big\{(\s,\p)
  \!\in\! \set \!\times\! \set \,\big|\, \s_i \!=\! \p_i \!\in\!
  \Xset^*_i\big(\v_i(\s_i)\big), \forall i \!\in\! \agt\big\}.
	\label{eq:eqpt}
\end{align}
Next, we establish $V$ as a size function (see
Definition~\ref{def:size_func}) for a superset of the DRoNE set. This
will be useful in proving convergence results in the sequel.

\begin{lem}\thmtitle{$V$ is a size function}
\label{lem:v_size_func}
Suppose Assumption~\ref{asmp:utility_bounded_gradient} holds. Define
$\ballb \ldef \times_{i \in \agt} \ball_{B_i}(\zero)$ and let
\begin{align}
  \phibt_i(\s_i, \p_i, \zetab_i) = \mu_i \zetab_i - \frac{1}{\lambda_i} [\s_i - \p_i],
  \quad \forall i \in \agt\,;
  \label{eq:phi_append}
\end{align}
where, $\forall i \in \agt$, $\s_i, \p_i \in \set_i$, and
$\zetab_i \in \ball_{B_i}(\zero)$.  Define,
\begin{align}
  \label{eq:z_set} & \Zset \ldef \Big\{(\s,\p,\zetab) \in \set \times \set \times \ballb \,\,\Big|\,\, \\ 
  \notag & \hspace{9em} \forall i \in \agt, \,\, \s_i \in \Xset^*_i\big(\phibt_i(\s_i, \p_i, \zetab_i)\big) \Big\}\,.
\end{align}
Then, $V\big(\s, \, \phibt(\s, \p, \zetab)\big)$ as a function of
$(\s, \p, \zetab)$ is a size function for $\Zset$. Moreover,
\begin{align*}
  \Big\{(\sst,\sst,\zetabst) \,\Big|\,
  \sst \in \ne(\game[DR](s)), \zetabst_i = \v_i(\sst), \forall i \in \agt \Big\}
  \subseteq \Zset\,.
\end{align*} 
\end{lem}
\begin{proof}
  First recall from~\eqref{eq:lyap} and the discussion following it
  that $V$ is Lipschitz and; hence, continuous. Moreover, recall that by
  definition, $V(\s,\phib) \geq 0$, $\forall \, \s \in \set$, and
  $\phib$.
  From~\eqref{eq:prox_derivative} and~\eqref{eq:phi}, suppose
  $\p \in \set$ (with components $\p_i \in \set_i$, $i \in \agt$)
  describes the state at the previous time step; then
\begin{align*}
  \phib_i(\s,\p) = \mu_i \v_i(\s) - \frac{1}{\lambda_i} [\s_i - \p_i],
  \quad \forall i \in \agt\,.
\end{align*}
Since $\forall i \in \agt$,
$\v_i(\s) \in \del_\x \big[ \U[DR](s)_i(\x,\s_{-i}) \big]
\big|_{\s_i}$, we can consider the evolution of \algoname
(Definition~\ref{def:isbrag}) with $\phibt$ in~\eqref{eq:phi_append}
in the set $\set \times \set \times \ballb$.  Because of the
assumption on the compactness of $\set_i$,
and Assumption~\ref{asmp:utility_bounded_gradient},
$\forall i \in \agt$,
$\phibt_i \in \ball_{\mu_i B_i + D_i/\lambda_i}(\zero)$,
$\zetab_i \in \ball_{B_i}(\zero)$, $\s_i \in \set_i$ and
$\p_i \in \set_i$. Then, using these definitions and with a
  slight abuse of notation 
  $V(\s,\p,\zetab) \ldef V(\s,\phibt)$, we have that $V(\s,\p,\zetab)$
  is continuous in all its arguments; since $\phibt$
  in~\eqref{eq:phi_append} is continuous in its arguments. For the
  sake of brevity, define $\Yset \ldef \set \times \set \times \ballb$,
  which is a compact set. Then, $V(\s,\p,\zetab)$ attains its maximum
  and minimum on $\Yset$, and it is easy to deduce that there is a
  continuous extension of $V$ (say $\Vt$) that is proper. For example,
\begin{align*}
	\Vt(\s,\p,\zetab) = 
	\begin{cases}
		V(\s,\p,\zetab), \qquad \text{if } (\s,\p,\zetab) \in \Yset; \\
		d\big((\s,\p,\zetab),\Yset\big) + V(\s^*,\p^*,\zetab^*) \quad \text{otherwise};
	\end{cases}\\
\end{align*}
where
$(\s^*,\p^*,\zetab^*) \in \argmin_{(\hat\s,\hat\p,\hat\zetab) \in
  \Yset} \|(\hat\s,\hat\p,\hat\zetab)-(\s,\p,\zetab)\|$. 

Now, $\Vt(\s,\p,\zetab) \geq 0$, $\forall (\s,\p,\zetab) \in
\Yset$. Moreover, $\Vt(\s,\p,\zetab) = 0$ iff $\forall i \in \agt$,
$\max_{\x_i \in \set_i } \x_i^\top \phibt_i(\s_i,\p_i,\zetab_i) =
\s_i^\top \phibt_i(\s_i,\p_i,\zetab_i)$, \emph{i.e.}
$\s_i \in \Xset^*_i\big(\phibt_i(\s_i, \p_i, \zetab_i)\big)$. Further,
since $\Vt$ is continuous, $\forall l \in \realnonneg$, the
$l$-sublevel set $\{(\s,\p,\zetab)\,|\, \Vt(\s,\p,\zetab) \leq l \}$
is a closed subset of $\real^n$. Now, by construction, these
$l$-sublevel sets are compact, since $V$ attains its maximum in
$\Yset$. Thus, by Definition~\ref{def:size_func}, $\Vt$ is a size
function for $\Zset$. Notice from~\eqref{eq:dyn_s} and~\eqref{eq:phi}
that the set $\Yset$ is further invariant under \algoname. Thus
equivalently, $V$ is a size function for $\Zset$.

The last claim on subset relationship can be easily verified
from~\eqref{eq:min_supergr} and the definition of a NE in
Definition~\ref{def:eps_ne}. This completes the proof.
\end{proof}

Before providing convergence guarantees for our algorithm, we
define a class of games for which the supergradients play well with
the proximal terms.
\begin{defn}\thmtitle{Amicable supergradients}
\label{def:amicable_supgr}
The game $\game[DR](s)$ (satisfying Standing
Assumption~\ref{asmp:gen_game_assumption}) is said to have
\emph{amicable supergradients} if the following holds. Consider an
arbitrary but fixed $i \in \agt$. Let $(\s_i,\s_{-i}) \in \set$. Then,
$\exists \, d_i \in \real_{>0}$ such that
$\forall (\sb_i,\sb_{-i}) \in \ball_{d_i}(\s_i,\s_{-i})$,
$\exists \, \zetab_i^1 \in \del_\x \big[ \U[DR](s)_i(\x,\s_{-i}) \big]
\big|_{\s_i}$,
$\exists \, \zetab_i^2 \in \del_\x \big[ \U[DR](s)_i(\x,\sb_{-i})
\big] \big|_{\s_i}$, and
$\exists \, \zetab_i^3 \in \del_\x \big[ \U[DR](s)_i(\x,\s_{-i}) \big]
\big|_{\sb_i}$, such that,
\begin{align*}
  & \ip{\big[\zetab_i^2 - \zetab_i^1\big]}{\big[\sb_i - \s_i\big]}
    - \ip{\big[\zetab_i^3 - \zetab_i^1\big]}{\big[\sb_i - \s_i\big]} \leq c_i \|\sb_i - \s_i\|^2\!\!,
\end{align*}
for some $c_i \in \real$. This $\{c_i\}_{i \in \agt}$ is referred to
as the \emph{factors of amicability} of the supergradients.
\bulletend
\end{defn}
The name \emph{amicable} is chosen to represent the fact that the
supergradients behave in a ``friendly'' manner with respect to the
proximal term in~\eqref{eq:prox_derivative}. The significance of the
previous definition will come into focus in the proof of the next
theorem, where we will use the factors of concavity
$\{\lambda_i^{-1}\}_{i \in \agt}$ to offset the factors of amicability
$\{c_i\}_{i \in \agt}$ and produce
monotonicity-like~\cite{YWC-CK-MA:24} behavior. In this regard,
``amicability'' can be seen to extend monotonicity.
Next, we state and prove conditions under which \algoname converges.

\begin{thm}\thmtitle{Convergence of \algoname}
\label{thm:isbrag_converge}
Suppose Assumption~\ref{asmp:utility_bounded_gradient}
holds. Moreover, suppose $\game[DR](s)$ has amicable supergradients
with factor of amicability $\{c_i\}_{i \in \agt}$. With
$d_{\mathsf{min}} \ldef \min_{i \in \agt} d_i$, choose
$\mu_i \in (0, \infty)$, $\forall i \in \agt$ and
\begin{align}
  \alpha_i \in \left(0, \,\, \min\left\{\frac{d_{\mathsf{min}}}{D_i},
  \,\, \frac{1}{2} \right\}\right], \quad \forall i \in \agt\,.
	\label{eq:alpha_choice}
\end{align}
Next, $\forall i \in \agt$, choose $\scalefactor_i \in (1,\infty)$ and
suppose
\begin{align}
	\lambda_i \in 
	\begin{cases}
          \displaystyle\left(\frac{1}{\scalefactor_i \mu_i c_i},
          \frac{1}{\mu_i c_i} \right), & \text{if } c_i > 0; \\\\
          \displaystyle\left(\frac{1}{\scalefactor_i \mu_i}, \infty\right), & \text{otherwise}\,.
	\end{cases}
	\label{eq:lambda_choice}
\end{align}
Define the size function
$\omega(\cdot) \ldef d(\cdot,\ne(\game[DR](s)))$ as the distance from
the NE set of $\game[DR](s)$. Denote
$\bar{\alpha} \ldef \max_{i \in \agt} \alpha_i$,
$\bar{\mu} \ldef \max_{i \in \agt} \mu_i$.  Finally, let $\s(t)$ be a
trajectory of the \algoname dynamics in Definition~\ref{def:isbrag}
from the initial condition
$\{\s_i(0) = \s_i(-1) \in \set_i\}_{i \in \agt}$.  Then, there exist
functions $\beta_1 \in \Mcal{KL}$, $\gamma_1 \in \Mcal{K}$, and a
continuous function
$\rho_1 : [0,0.5] \times \realnonneg \to \realnonneg$ such that
\begin{align}
  \omega\big(\s(t)\big) \leq \max\left\{\beta_1\Big(\omega\big(\s(0)\big),t\Big), \,\,
  \gamma_1\Big(K \rho_1\big(\bar{\alpha},\bar{\mu}\big)\Big)\right\}\,,
	\label{eq:dyn_iss_const_ip}
\end{align} 
where, the constant $K$ depends only on
$\{D_i, B_i, c_i, \scalefactor_i\}_{i \in \agt}$.  Moreover, $\rho_1$
is strictly increasing wrt both its arguments and $\rho_1(x,y) \to 0$
as $(x,y) \to 0$.
\end{thm}

\begin{proof}
  First, notice that because of Lemmas~\ref{lem:size_func}
  and~\ref{lem:v_size_func}, there exists functions
  $\sigma_1, \sigma_2 \in \Mcal{K}_\infty$ such that
\begin{align}
  \sigma_1(\omega(\s)) \leq V(\s)
  \leq \sigma_2(\omega(\s)), \quad \forall \s \in \set.
  \label{eq:v_upper_lower_k_infty} 
\end{align}
This can be seen from the statement of Lemma~\ref{lem:v_size_func} and realizing that  the set $\{(\s,\zetab)\,|\, \forall i \in \agt, \s_i \in \set_i, \zetab_i \in \del_{\x_i} [\U[DR](s)_i(\x_i,\s_{-i})]|_{\s_i}\}$ is invariant under \algoname.

Now, to bound the change in $V$, we first provide an upper bound on
$[\s_i(t+1) - \s_i(t)]^\top [\phib_i(t+1) - \phib_i(t)]$,
$\forall i \in \agt$. With a slight abuse of notation, set
$\s_i = \s_i(t)$, $\sp_i = \s_i(t+1)$, $\sm_i = \s_i(t-1)$,
$\s_{-i} = \s_{-i}(t)$, $\sp_{-i} = \s_{-i}(t+1)$, $\phib_i^- = \phib_i(t-1)$
$\phib_i = \phib_i(t)$, and $\phibp_i = \phib_i(t+1)$.

Consider an arbitrary but fixed agent $i \in \agt$. By the property of
supergradients, at time $t$, we have, $\forall \x \in \set_i$,
\begin{align}
  \label{eq:supergr_t} & \mu_i\U[DR](s)_i(\x,\s_{-i}) - \frac{1}{2\lambda_i}
                         \|\x - \sm_i\|^2 \leq \mu_i\U[DR](s)_i(\s_i,\s_{-i}) \\
  \notag & \qquad - \frac{1}{2\lambda_i} \|\s_i - \sm_i\|^2
           + \ip{\phib_i}{[\x - \s_i]} - \frac{1}{2\lambda_i} \|\x - \s_i\|^2;
\end{align}
and at time $t+1$, we have, $\forall \x \in \set_i$,
\begin{align}
  \label{eq:supergr_t_1} & \mu_i\U[DR](s)_i(\x,\sp_{-i}) -
                           \frac{1}{2\lambda_i} \|\x - \s_i\|^2 \leq \mu_i\U[DR](s)_i(\sp_i,\sp_{-i}) \\
  \notag & \quad - \frac{1}{2\lambda_i}
           \|\sp_i - \s_i\|^2 + \ip{\phibp_i}{[\x - \sp_i]} - \frac{1}{2\lambda_i} \|\x - \sp_i\|^2\,.
\end{align}
Note that the last quadratic terms in both of the previous two equations 
come from the strong concavity of the proximal term
$\|\cdot\|^2$~\cite{SB-LV:04}.  Substituting $\x = \sp_i$
in~\eqref{eq:supergr_t}, $\x = \s_i$ in~\eqref{eq:supergr_t_1}; and
combining the inequalities, we get
\begin{align*}
& \ip{[\phibp_i-\phib_i]}{[\sp_i - \s_i]} \\ 
& \leq \frac{1}{2\lambda_i} \|\sp_i - \sm_i\|^2 - \frac{1}{2\lambda_i} \|\s_i - \sm_i\|^2 - \frac{1}{2\lambda_i} \|\sp_i - \s_i\|^2 \\
& \qquad - \bigg[\frac{1}{2\lambda_i} + \frac{1}{2\lambda_i} \bigg] \|\sp_i - \s_i\|^2  +\mu_i \bigg[\U[DR](s)_i(\s_i,\s_{-i})\\
  & \qquad - \U[DR](s)_i(\sp_i,\s_{-i}) + \U[DR](s)_i(\sp_i,\sp_{-i}) - \U[DR](s)_i(\s_i,\sp_{-i})\bigg].
\end{align*}
Noting that
$\|\sp_i - \sm_i\|^2 = \|\sp_i - \s_i\|^2 + 2\ip{[\sp_i - \s_i]}{[\s_i
  - \sm_i]} + \|\s_i - \sm_i\|^2$, and applying appropriate algebraic
simplifications, leads to
\begin{align*}
  & \ip{[\phibp_i-\phib_i]}{[\sp_i - \s_i]} \\ 
& \leq \frac{1}{\lambda_i} \ip{[\sp_i - \s_i]}{[\s_i - \sm_i]} - \frac{1}{\lambda_i} \|\sp_i - \s_i\|^2 \\
& \qquad \qquad + \mu_i \Big[\U[DR](s)_i(\s_i,\s_{-i}) - \U[DR](s)_i(\sp_i,\s_{-i}) \\
& \qquad \qquad + \U[DR](s)_i(\sp_i,\sp_{-i}) - \U[DR](s)_i(\s_i,\sp_{-i}) \Big] \,.
\end{align*}
Now, suppose
$\zetab_i \big(\sp_i,\s_{-i}\big) \in \del_\x \Big[
\U[DR](s)_i(\x,\s_{-i}) \Big] \Big|_{\sp_i}$,
$\zetab_i \big(\s_i,\sp_{-i}\big) \in \del_\x \Big[
\U[DR](s)_i(\x,\sp_{-i}) \Big] \Big|_{\s_i}$.  Then using the appropriate
definition of the supergradients (see Definition~\ref{def:subgr}), we modify the previous bound on $\ip{[\phibp_i-\phib_i]}{[\sp_i - \s_i]}$ as
\begin{align*}
& \ip{[\phibp_i-\phib_i]}{[\sp_i - \s_i]} \\ 
& \leq \frac{1}{\lambda_i} \ip{[\sp_i - \s_i]}{[\s_i - \sm_i]} - \frac{1}{\lambda_i} \|\sp_i - \s_i\|^2 \\
& \qquad + \mu_i \ip{\Big[\zetab_i \big(\s_i,\sp_{-i}\big) - \zetab_i \big(\sp_i,\s_{-i}\big)\Big]}{\big[\sp_i - \s_i\big]}\\
& = \frac{1}{\lambda_i} \ip{[\sp_i - \s_i]}{[\s_i - \sm_i]} - \frac{1}{\lambda_i} \|\sp_i - \s_i\|^2 \\
& \qquad + \mu_i \ip{\Big[\zetab_i \big(\s_i,\sp_{-i}\big) - \zetab_i \big(\s_i,\s_{-i}\big)\Big]}{\big[\sp_i - \s_i\big]} \\
& \qquad - \mu_i \ip{\Big[\zetab_i \big(\sp_i,\s_{-i}\big) - \zetab_i \big(\s_i,\s_{-i}\big)\Big]}{\big[\sp_i - \s_i\big]}\,.
\end{align*}
Here, for the last equality, we have
used the linearity property of inner products. Then, from 
  Definition~\ref{def:amicable_supgr},
\begin{align*}
& \ip{[\phibp_i-\phib_i]}{[\sp_i - \s_i]} \leq \\
&  \, \frac{1}{\lambda_i} \ip{[\sp_i - \s_i]}{[\s_i - \sm_i]} - \frac{1}{\lambda_i} \|\sp_i - \s_i\|^2 + \mu_i c_i \|\sp_i - \s_i\|^2 \,.
\end{align*}
Indeed, we can use the inequality in
Definition~\ref{def:amicable_supgr}, since we have chosen
$\alpha_i < d_{\mathsf{min}}/D_i$ to ensure that the opponents'
strategies remain within the $d_i$-ball around the current strategy
profile.  Now, because of~\eqref{eq:lambda_choice},
$(- \frac{1}{\lambda_i} + \mu_i c_i) \|\sp_i - \s_i\|^2 \leq 0$. Then,
substituting in the dynamics~\eqref{eq:dyn_s} in the previous
inequality gives us
\begin{align*}
  \notag & \ip{[\phibp_i-\phib_i]}{[\sp_i - \s_i]} \leq \frac{1}{\lambda_i} \ip{[\sp_i - \s_i]}{[\s_i - \sm_i]} \\
         & \quad = \frac{\alpha_i^2}{\lambda_i} \ip{[\suppf_i(\phib_i) - \s_i]}{[\suppf_i(\phib_i^-) - \sm_i]} \,.
\end{align*}
Recall that $\suppf_i(\phib_i), \suppf_i(\phib_i^-) \in
\set_i$. Relying on the Cauchy-Schwarz inequality on the previous step
and utilizing the fact that $D_i$ is the diameter of $\set_i$, we
conclude that
\begin{align}
  \notag & \ip{[\phibp_i-\phib_i]}{[\sp_i - \s_i]} \leq \\
  & \quad \frac{\alpha_i^2}{\lambda_i} \|\suppf_i(\phib_i) - \s_i\| \|\suppf_i(\phib_i^-) - \sm_i\| \leq \alpha_i^2\frac{D_i^2}{\lambda_i} \,.
           \label{eq:v_change_ip}
\end{align}

Finally, we provide an upper bound for $\|\phibp_i - \phib_i\|$.
Quickly, define $\forall i \in \agt$, 
\begin{align*}
	\tilde{c}_i = 
	\begin{cases}
		c_i, & \text{if } c_i > 0;\\
		1, & \text{otherwise}\,;
	\end{cases}
\end{align*}
and notice that $\frac{1}{\lambda_i} < \scalefactor_i \mu_i \tilde{c}_i$, $\forall i \in \agt$.
Then, using the properties of $\|\cdot\|$ and the triangle inequality,
\begin{align}
	\label{eq:v_change_norm}  \|\phibp_i - \phib_i\| & \leq \mu_i\|\v_i^+ -\v_i\|  + \frac{1}{\lambda_i} \big\|\sp_i - \sm_i \big\|\\
	\notag & \leq 2 \mu_i B_i + \frac{D_i}{\lambda_i}\leq (2B_i + D_i \scalefactor_i \tilde{c}_i) \mu_i \rdef A_i \mu_i  \,.
\end{align}
Recall that $\bar{D}_i \geq D_i$, $\forall i \in \agt$. Now, combining~\eqref{eq:lyap_change} from Lemma~\ref{lem:lyap_change},~\eqref{eq:v_change_ip}, and~\eqref{eq:v_change_norm}, gives us,
\begin{align}
	\label{eq:v_change_final_st1} & V\big(\s(t+1),\phib(t+1)\big) - V\big(\s(t),\phib(t)\big) \\
	 \notag & \leq \sum_{i \in \agt} \Big[ [1-\alpha_i]\alpha_i\frac{D_i^2}{\lambda_i} + \mu_i A_i \bar{D}_i \Big] - \underline{\alpha} V(\s(t),\phib(t)) \\
	 \notag & \leq - \underline{\alpha} V(\s(t),\phib(t)) + K \rho\big(\bar{\alpha},\bar{\mu}\big)
\end{align}
where, $\forall i \in \agt$,
\begin{align}
  & K_i \ldef \max \Big\{ \frac{D_i^2}{\lambda_i}, (2B_i + D_i \scalefactor_i \tilde{c}_i)\bar{D}_i \Big\},
	\label{eq:k_def}
\end{align}
the constant $K \ldef \max_{i \in \agt} n K_i$, and the function
${\rho_1 : [0,0.5] \times \realnonneg \to \realnonneg}$ is defined as
\begin{align}
	\rho_1(x,y) = x- x^2 +  y \,.
	\label{eq:f_def}
\end{align}
Observe that the last inequality in~\eqref{eq:v_change_final_st1} holds
since ${x - x^2 + 1}$ is strictly increasing in $[0,0.5]$ and as such
${\alpha_i - \alpha_i^2 + 1} \leq {\bar{\alpha}_i - \bar{\alpha}_i^2 +
  1}$, $\forall i \in \agt$. Moreover, from
Lemma~\ref{lem:lyap_change}, $\underline{\alpha} > 0$.  Thus,
comparing~\eqref{eq:v_upper_lower_k_infty}
and~\eqref{eq:v_change_final_st1} with~\eqref{eq:iss_v_bound}
and~\eqref{eq:iss_v_change}, we conclude that the system is ISS with
input $K
\rho_1\big(\bar{\alpha},\bar{\mu}\big)$. Then,~\eqref{eq:dyn_iss_const_ip}
follows immediately from Definition~\ref{def:iss}. Moreover, the
properties of $\rho_1$ are immediate from its definition
in~\eqref{eq:f_def}.
\end{proof}

From the previous result, we see that, given the right choice of
parameters, \algoname converges to a region around the DRoNE set. We
conclude this section by discussing the convergence bound in the
following remark.

\begin{rem}\thmtitle{On the effect of parameters for \algoname}
\label{rem:isbrag_param}
\rm From Theorem~\ref{thm:isbrag_converge}, it is clear that the
agents' choices of $\{\alpha_i,\mu_i,\lambda_i\}_{i \in \agt}$ affect
not only the rate of convergence of \algoname (through $\beta_1$), but
also the region to which it converges (through
$\gamma_1$). This region can be made
  arbitrarily small due to the properties of $\gamma_1$ and $\rho_1$
  in~\eqref{eq:dyn_iss_const_ip}. From~\eqref{eq:v_change_final_st1}, it is easy to see
that since $\underline{\alpha} = \min_{i \in \agt} \alpha_i$, the rate
of convergence depends on the agent that has the smallest step
size. If this decreases, the rate of convergence decreases.

Moreover, $\underline{\alpha}$ affects the region to which \algoname
converges as follows. From~\cite{ZJ-ES-YW:99}, we can characterize
$\gamma_1 \in \Mcal{K}$ from Theorem~\ref{thm:isbrag_converge} as
\begin{align}
  \gamma_1\Big(K \rho\big(\bar{\alpha},\bar{\mu}\big)\Big) =
  \sigma_1^{-1}\left( \frac{K}{\mathfrak{M} \underline{\alpha}} \rho_1\big(\bar{\alpha},\bar{\mu}\big) \right)\,,
	\label{eq:gamma_characterization}
\end{align}
where $\mathfrak{M}$ is any number in $(0,1)$ and $\sigma_1$ comes
from~\eqref{eq:v_upper_lower_k_infty}. Thus, using the relation
in~\eqref{eq:f_def}, agent $i \in \agt$ can choose $\alpha_i$ and
$\mu_i$ arbitrarily small and subsequently make
$\gamma_1\big(K \rho_1(\bar{\alpha},\bar{\mu})\big)$ arbitrarily
small.

In light of this, it might be tempting to think that by making $\mu_i$
arbitrarily small, agent $i \in \agt$ diminishes the effect of its
supergradient and is dominated by the inertial term (see
Remark~\ref{rem:phi_choice}). However, due to the
constraint in~\eqref{eq:lambda_choice}, reducing $\mu_i$ requires
increasing $\lambda_i$ which in turn requires decreasing $1/\lambda_i$
too. Hence the parameters $\mu_i$ and $1/\lambda_i$ work in tandem to
counteract the effect of $c_i$ and preserve supergradient information.
\bulletend
\end{rem}

\section{Distributed DRoNE Seeking} \label{sec:dist_algo} To implement
the \algoname dynamics, each agent $i \in \agt$ requires knowledge of
$\s_{-i}(t)$, the current strategies of all other agents. 
In this section, we begin to adapt \algoname to a distributed
communication network $\grph = (\agt,\edg)$ and solve
Problem~\ref{prob}:~\eqref{prob:distributed_drones} under
Assumptions~\ref{asmp:ind_uncertain} and~\ref{asmp:ind_sample} on
individual uncertainty. We tackle the problem in its full generality
in Section~\ref{sec:dist_drones_shared}.

In order to compensate for the lack of knowledge of others'
strategies, we will allow the agents to estimate the same through a
consensus-like protocol. This will induce an error (say
$\Deltab_i(t)$) in the computation of $\phib_i(t)$ (specifically in
the computation of $\v_i(t)$ using~\eqref{eq:min_supergr}) for agent
$i \in \agt$, $\forall t \in \intnonneg$. Let $\phibh(t)$, with
$\phibh_i(t)$ components corresponding to $i \in \agt$, denote the
perturbed version of $\phib(t)$ at time $t \in \intnonneg$. Then, we
update the dynamics in~\eqref{eq:dyn_s} appropriately as,
$\forall i \in \agt$,
\begin{subequations}
\begin{align}
  \label{eq:dyn_phi_hat} \phibh_i(t)
  & = \mu_i \big[\v_i(\s(t)) + \Deltab_i(t) \big] + \w_i(\s_i(t),t) \\
	\notag & = \phib_i(t) + \mu_i \Deltab_i(t) \,,\\
  \label{eq:dyn_pi_hat} \hat{\pib}_i(t)
  & = \suppf_i\big(\phibh_i(t)\big) = \suppf_i\big(\phib_i(t)\big) + \Deltabpr_i(t), \\
  \label{eq:dyn_s_hat} \s_i(t+1)
  & = \s_i(t) + \alpha_i \, \Big[\hat{\pib}_i(t) - \s_i(t) \Big]
  \\
	\notag 
& = \s_i(t) + \alpha_i \, \Big[\suppf_i\big(\phib_i(t)\big) - \s_i(t) \Big] + \alpha_i\Deltabpr_i(t)\,.
\end{align}
\label{eq:dyn_hat}
\end{subequations}
Here, $\Deltabpr_i(t)$ is the error in the strategy updates while
$\Deltab_i(t)$ is the error in the calculation of supergradients.
Note that, if the agents estimate others' strategies and compute
$\phibh_i(t)$ according to the estimated strategies, then, because
of~\eqref{eq:min_supergr},
$\v_i(\s(t)) + \Deltab_i(t) \in \del_\x [ \U[DR](s)_i(\x,\s_{-i}) ]
|_{\s_i}$ and hence $\|\Deltab_i(t)\| \leq 2B_i$,
$\forall t \in \intnonneg$. Moreover, because
of~\eqref{eq:argmax_x_star}, and~\eqref{eq:supp_func},
$\hat{\pib}_i(t) = \suppf_i\big(\phib_i(t)\big) + \Deltabpr_i(t) \in
\set_i$. Hence, since $\suppf_i\big(\phib_i(t)\big) \in \set_i$,
$\|\Deltabpr_i(t)\| \leq D_i$.  Thus, by properly choosing the
$\{\alpha_i, \mu_i\}_{i \in \agt}$ parameters,
the effect of the disturbance can be mitigated
accordingly. To show this formally, we first adapt
Lemma~\ref{lem:lyap_change} to account for the error induced by
estimations.
\begin{lem}\thmtitle{Bound on Lyapunov function differences in
   the presence of disturbances}
\label{lem:lyap_change_perturb}
Consider the dynamics~\eqref{eq:dyn_s_hat} with input sequence
$\{\Deltab(t), \Deltabpr(t)\}_{t \in \intnonneg}$ from an initial
condition $\{\s_i(0) \in \set_i\}_{i \in \agt}$. Suppose
$\{\phib(t)\}_{t \in \intnonneg}$ satisfies $\|\phib_i(t)\| \leq M_i$
(for some $M_i \in \real_{>0}$), $\forall i \in \agt$,
$\forall t \in \intnonneg$. Take any $\bar{D}_i > D_i$,
$\forall i \in \agt$.
Denote $\bar{\alpha} \ldef \max_{i \in \agt} \alpha_i$,
$\bar{\mu} \ldef \max_{i \in \agt} \mu_i$. Then,
$\forall t \in \intnonneg$, with
$\underline{\alpha} = \min_{i \in \agt} \alpha_i$,
\begin{align}
  \label{eq:lyap_change_perturb}
  & V\big(\s(t+1),\phib(t+1)\big) - V\big(\s(t),\phib(t)\big) \leq \\
  \notag
  & \sum_{i \in \agt} \Big[ \frac{1-\alpha_i}{\alpha_i}[\s_i(t+1) - \s_i(t)]^\top [\phib_i(t+1) - \phib_i(t)] \\ 
  \notag & + \bar{D}_i \|\phib_i(t+1) - \phib_i(t)\| \Big]
           - \underline{\alpha} V(\s,\phib) + \rho_2(\bar{\alpha},\bar{\mu})\|\Deltabpr(t)\|.
\end{align}
Here, the term $\rho_2(\bar{\alpha},\bar{\mu})$ is characterized as
\begin{align}
  \rho_2(\bar{\alpha},\bar{\mu}) = \max_{i \in \agt}[A_i-B_i] \, \bar{\alpha} \, \bar{\mu} + \max_{i \in \agt}A_i \, \bar{\mu}\,.
	\label{eq:rho_2}
\end{align}
Further, $A_i$ is as in~\eqref{eq:v_change_norm} and $B_i$ is as in
Assumption~\ref{asmp:utility_bounded_gradient}.

\end{lem}   
\begin{proof}
  As before, with a slight abuse of notation, we set $\s = \s(t)$,
  $\sp = \s(t+1)$,$\phib = \phib(t)$, $\phibp = \phib(t+1)$,
  $\phibh = \phibh(t)$, $\phibhp = \phibh(t+1)$,
  $\Deltab = \Deltab(t)$, and $\Deltabpr = \Deltabpr(t)$. Then, from
  the definition of $V$ in~\eqref{eq:lyap} and the dynamics
  in~\eqref{eq:dyn_s_hat}, we have that
\begin{align*}
	& V(\sp,\phibp) - V(\s,\phib) \\
	& = \sum_{i \in \agt} \Big[ \max_{\x_i \in \set_i } \x_i^\top \phibp_i - \max_{\x_i \in \set_i } \x_i^\top \phib_i \\
	& - \s_i^\top \big[ \phibp_i - \phib_i \big] - \big[\sp_i - \s_i \big]^\top \big[ \phibp_i - \phib_i \big] - \big[\sp_i - \s_i \big]^\top\phib_i \Big] \\
	& \leq \sum_{i \in \agt} \Big[ \ip{\suppf_i(\phib_i)}{[\phibp_i - \phib_i]} + \bar{D}_i \|\phibp_i - \phib_i\| \\
	& \qquad - \ip{\s_i}{\big[ \phibp_i - \phib_i \big]} - \alpha_i \ip{\big[\suppf_i(\phibh_i) - \s_i \big]}{\big[ \phibp_i - \phib_i \big]} \\ 
	& \qquad - \alpha_i \ip{\big[\suppf_i(\phibh_i) - \s_i \big]}{\phib_i} \Big] \,.
\end{align*}
This chain of inequalities follow the same
arguments as in the proof of Lemma~\ref{lem:lyap_change} with the help
of Lemma~\ref{lem:change_max_lin_func} and the
dynamics~\eqref{eq:dyn_hat} (specifically~\eqref{eq:dyn_pi_hat} and~\eqref{eq:dyn_s_hat}).
Now, by appropriately substituting in~\eqref{eq:dyn_pi_hat} in the previous inequality, we get
\begin{align*}
	& V(\sp,\phibp) - V(\s,\phib) \\
	& \leq \sum_{i \in \agt} \Big[ \ip{\big[\suppf_i (\phibh_i) - \Deltabpr_i\big]}{[\phibp_i - \phib_i]} + \bar{D}_i \|\phibp_i - \phib_i\| \\
	& \qquad - \ip{\s_i}{\big[ \phibp_i - \phib_i \big]} - \alpha_i \ip{\big[\suppf_i(\phibh_i) - \s_i \big]}{\big[ \phibp_i - \phib_i \big]} \\ 
	& \qquad - \alpha_i \ip{\big[\suppf_i(\phib_i) + \Deltabpr_i - \s_i \big]}{\phib_i} \Big] \,.
\end{align*}
Now to upper bound this term, we use similar
techniques as in the proof of
Lemma~\ref{lem:lyap_change}. Specifically, notice that that
$\sum_{i \in \agt}\big[\suppf_i(\phib_i) - \s_i \big]^\top\phib_i =
V(\s,\phib)$. Then by combining appropriate terms from the previous step, separating out
the terms involving $\Deltabpr_i$'s and further, using Cauchy-Schwarz inequalities we get
$\ip{[-\Deltabpr_i]}{[\phibp_i - \phib_i]} \leq
\|\Deltabpr_i\|\|\phibp_i - \phib_i\|$ and
$\ip{[-\Deltabpr_i]}{\phib_i} \leq \|\Deltabpr_i\|\|\phib_i\|$. Putting all of this together gives us
\begin{align*}
	& V(\sp,\phibp) - V(\s,\phib) \\
	& \leq \sum_{i \in \agt} \Big[ [1-\alpha_i] \ip{[\suppf_i(\phibh_i) - \s_i]}{[\phibp_i - \phib_i]} + \bar{D}_i \|\phibp_i - \phib_i\| \Big] \\
	& \qquad + \sum_{i \in \agt} \Big[ \|\Deltabpr_i\|\|\phibp_i - \phib_i\| + \alpha_i\|\Deltabpr_i\|\|\phib_i\|\Big] - \underline{\alpha} V(\s,\phib)\,.
\end{align*}
Finally, rearranging~\eqref{eq:dyn_s_hat} in the previous inequality gives
us all the terms in~\eqref{eq:lyap_change_perturb} except the last
term involving $\|\Deltabpr(t)\|$. To obtain the latter, we use the
upper bounds from~\eqref{eq:v_change_norm} and the discussion
after~\eqref{eq:phi}; then, use similar arguments as in the proof of
Theorem~\ref{thm:isbrag_converge}. This completes the proof.
\end{proof}

Note that from~\eqref{eq:rho_2}, 
the gain $\rho_2$
in~\eqref{eq:lyap_change_perturb} can be made arbitrarily small by
choosing arbitrarily small parameters
$\{\alpha_i,\mu_i\}_{i \in \agt}$. Thus,
$\rho_2(\bar{\alpha},\bar{\mu})\|\Deltabpr(t)\|$
in~\eqref{eq:lyap_change_perturb} is a class-$\Mcal{K}$ function of
$\|\Deltabpr(t)\|$. Hence,~\eqref{eq:lyap_change_perturb} is almost
similar to the ISS requirement in~\eqref{eq:iss}. We can handle the
extra terms like we did in the proof of
Theorem~\ref{thm:isbrag_converge} and show practical convergence when
agents need to estimate others' strategies through some algorithm. We
do this next.

Before proposing the distributed algorithm in full, we provide details
about the dynamic consensus protocol that will be used to this end. In
this manuscript, we adopt the protocol where the agents are allowed to
communicate with each other during $T \in \intpos$ intermediate rounds
to infer other agents' strategies. 
As discussed in~\cite{SSK-BVS-JC-RAF-KML-SM:18-csm}, we employ a
dynamic average consensus algorithm that produces a better tracking
response that a static consensus counterpart. 
The specific algorithm we use from~\cite{SSK-BVS-JC-RAF-KML-SM:18-csm}
is detailed next,
\begin{subequations}
\begin{align}
  \notag & \vcon_i\Big(t + \frac{\tau + 1}{T}\Big) = b_1 b_2 b_3 \sum_{j \in \neigh_i} a_{ij} \Big[\xcon_i\Big(t + \frac{\tau}{T}\Big) \\
         & \hspace{8em} - \xcon_j\Big(t + \frac{\tau}{T}\Big) \Big] + \vcon_i\Big(t + \frac{\tau}{T}\Big),\\
  \notag & \zcon_i\Big(t + \frac{\tau + 1}{T}\Big) = [1 - b_1 b_2] \zcon_i\Big(t + \frac{\tau}{T}\Big) - b_1 \vcon_i\Big(t + \frac{\tau}{T}\Big) \\
         & \hspace{2em} - b_1 b_3 \sum_{j \in \neigh_i} a_{ij} \Big[\xcon_i\Big(t + \frac{\tau}{T}\Big) - \xcon_j\Big(t + \frac{\tau}{T}\Big) \Big],\\
         & \xcon_i\Big(t + \frac{\tau}{T}\Big) = \zcon_i\Big(t + \frac{\tau}{T}\Big) + u_i\Big(t + \frac{\tau}{T}\Big)\,.
\end{align}
	\label{eq:consensus_dyn}
      \end{subequations} 
      Here, $\{\xcon_i, \vcon_i, \zcon_i\}_{i \in \agt}$ are internal
      variables of the dynamics and $b_1, b_2, b_3 \in \realpos$ are
      constants that can be chosen appropriately to ensure
      that~\eqref{eq:consensus_dyn} asymptotically tracks
      $[1/n]\sum_{i \in \agt} u_i(t)$ with a small error. This error
      depends on the aforementioned constants $b_1, b_2, b_3$, the
      eigenvalues of the graph Laplacian associated with $\grph$, and
      (most importantly) an upper bound on
      $|u_i(t+[\tau+1]/T) - u_i(t+\tau/T)|$. We refer the reader
      to~\cite[Theorem S2]{SSK-BVS-JC-RAF-KML-SM:18-csm} for details
      on how to tune the parameters of~\eqref{eq:consensus_dyn} and
      minimize the error.

      Now, in order to make the notations cleaner, we encapsulate the
      states in~\eqref{eq:consensus_dyn} into
      $\chib_i \ldef [\xcon_i, \vcon_i, \zcon_i]^\top$ for each agent
      $i \in \agt$. This allows us to rewrite~\eqref{eq:consensus_dyn}
      as
\begin{align}
	\chib_i\Big(t + \frac{\tau + 1}{T}\Big) = \f_i\bigg(\Big\{\chib_j\Big(t + \frac{\tau}{T}\Big)\Big\}_{j \in \neighb_i}, u_i\Big(t + \frac{\tau}{T}\Big) \bigg)\,.
	\label{eq:consensus_dyn_combine}
\end{align}
Observe that~\eqref{eq:consensus_dyn_combine} is a fully distributed
algorithm that requires agents to only pass information between
neighbors. Moreover, for the sake of brevity, let $\proj_x(\chib_i)$
denote the $\xcon_i$ component of $\chib_i$. Now, for agent
$i \in \agt$ to approximate $\s_{-i}$, it has to run multiple copies
of~\eqref{eq:consensus_dyn_combine} across the components of
$\s_{-i}$. To that end, let $\sh_{ij} \in \real^{n_j}$ (with
components $\hat{s}_{ij,l}$) be agent $i$'s estimate of $\s_j$. We
collect this into a stacked vector called $\hat{\s_{-i}}$ for each
$i \in \agt$. Similarly, let $\chib_{ij}^l$ denote the dynamic
consensus vector associated with $\hat{s}_{ij,l}$. Combining all of
this, we present the distributed algorithm that partially solves
Problem~\ref{prob}:~\eqref{prob:distributed_drones} next and discuss
it in the remark that follows.
\begin{defn}\thmtitle{d-\algoname}
\label{def:disbrag}
We refer to the following algorithm as \emph{Distributed Inertial Supported Better Response Ascending superGradient dynamics} or \emph{d-\algoname}. 
\begin{subequations}
\begin{align}	
	\label{eq:disbrag_s_update} & \s_i(t+1) = \s_i(t) + \alpha_i \, \bigg[\suppf_i\bigg(\phibh_i\Big(t + \frac{T-1}{T}\Big)\bigg) - \s_i(t) \bigg], \\
	\label{eq:disbrag_phi_hat} & \phibh_i\Big(t + \frac{\tau}{T}\Big) = \mu_i \v_i\bigg(\s_i(t), \hat{\s_{-i}}\Big(t + \frac{\tau}{T}\Big) \bigg) + \w_i(\s_i(t),t), \\
	\label{eq:disbrag_chi_ij} & \chib_{ij}^l\Big(t + \frac{\tau + 1}{T}\Big) = \f_i\bigg(\Big\{\chib_{kj}^l\Big(t + \frac{\tau}{T}\Big)\Big\}_{k \in \neighb_i}, 0 \bigg), \\
	\label{eq:disbrag_chi_ii} & \chib_{ii}^l\Big(t + \frac{\tau + 1}{T}\Big) = \f_i\bigg(\Big\{\chib_{ki}^l\Big(t + \frac{\tau}{T}\Big)\Big\}_{k \in \neighb_i}, n \, s_{i,l}(t) \bigg), \\
	\label{eq:disbrag_s_hat} & \hat{s}_{ij,l}\Big(t + \frac{\tau}{T}\Big) = \proj_x\bigg( \chib_{ij}^l\Big(t + \frac{\tau}{T}\Big)\bigg) \,.
\end{align}
\label{eq:disbrag}
\end{subequations}
The previous equation holds $\forall i,j \in \agt$, $\forall \, t, \tau \in \intnonneg$. The index $l$ appropriately belongs to $\{1,\cdots,n_i\}$. The $\v_i$ in~\eqref{eq:disbrag_phi_hat} is computed similarly as in~\eqref{eq:min_supergr} but with $\U[DR]_i$ instead of $\U[DR](s)_i$.
\bulletend
\end{defn}

\begin{rem}\thmtitle{On d-\algoname}
\rm
From~\eqref{eq:disbrag}, it is easy to notice the two time-scale approach in which d-\algoname proceeds. Each agent updates its strategy using~\eqref{eq:disbrag_s_update} at every $t \in \intpos$; while every agent runs the dynamic consensus protocol using~\eqref{eq:disbrag_chi_ij}-\eqref{eq:disbrag_s_hat} for $T$ sub (steps) $\{t,t+1/T,\cdots,t+[T-1]/T\}$. Moreover, agent $i\in \agt$ performs the update in~\eqref{eq:disbrag_s_update} at $t+1$ by computing $\phibh_i$ using the most recent estimate of $\hat{\s_{-i}}$ at $t+[T-1]/T$ (compare~\eqref{eq:disbrag_s_update} and~\eqref{eq:disbrag_phi_hat}).\\
Recall that $s_{i,l}$ is the $l\tth$ component of $\s_i$.
Hence, regarding the consensus update, notice from~\eqref{eq:disbrag_chi_ij} and~\eqref{eq:disbrag_chi_ii} that agent $i \in \agt$ contributes $n\,s_{i,l}$ to $\chib_{ii}^l$ and $0$ to all other $\chib_{ij}^l$, $j \in \agt \setminus \{i\}$. This in fact ensures that $\chib_{ji}^l(t)$ asymptotically tracks $s_{i,l}(t)$, $\forall i,j \in \agt$.    	
\bulletend
\end{rem}

We conclude this section by providing convergence guarantees for d-\algoname.
\begin{thm}\thmtitle{Convergence of d-\algoname}
\label{thm:disbrag_converge}
Suppose Assumptions~\ref{asmp:ind_uncertain},~\ref{asmp:ind_sample},~\ref{asmp:network_conn} , and~\ref{asmp:utility_bounded_gradient} hold. Moreover, suppose $\game[DR]$ has amicable supergradients with factor of amicability $\{c_i\}_{i \in \agt}$. With $d_{\mathsf{min}} \ldef \min_{i \in \agt} d_i$, choose $\mu_i \in (0, \infty)$, $\forall i \in \agt$ and suppose $\forall i \in \agt$, $\alpha_i$ satisfies~\eqref{eq:alpha_choice} and $\lambda_i$ satisfies~\eqref{eq:lambda_choice}.
Define the size function $\omega(\cdot) \ldef d(\cdot,\ne(\game[DR]))$ as the distance from the NE set of $\game[DR]$. Denote $\bar{\alpha} \ldef \max_{i \in \agt} \alpha_i$, $\bar{\mu} \ldef \max_{i \in \agt} \mu_i$. Suppose for~\eqref{eq:consensus_dyn}, $\vcon_i(0) = 0$, $\zcon_i(0) \in \real$, $\forall i \in \agt$.
Further, let $K_i$ be as in~\eqref{eq:k_def} with $K = \max_{i \in \agt}n K_i$; $\rho_1$ be as in~\eqref{eq:f_def}, and $\rho_2$ be as in~\eqref{eq:rho_2}.
Finally, let $\s(t)$ be the solution to the dynamics in Definition~\ref{def:disbrag} from initial condition $\{\s_i(0) = \s_i(-1) \in \set_i\}_{i \in \agt}$, with appropriately chosen parameters $b_1$, $b_2$, $b_3$, and $T$. \\
Then, there exists functions $\beta_2 \in \Mcal{KL}$, $\gamma_2 \in \Mcal{K}$ such that
\begin{align}
	\label{eq:dyn_iss_w_consensus} & \omega\big(\s(t)\big) \leq \max\Big\{\beta_2\Big(\omega\big(\s(0)\big),t\Big), \\
	 \notag & \hspace{10em} \gamma_2\Big(K \rho_1\big(\bar{\alpha},\bar{\mu}\big) + J \rho_2\big(\bar{\alpha},\bar{\mu}\big)\Big)\Big\}\,,
\end{align} 
where, $J$ is such that $\max_{\x \in \set}\|\x\| \leq J$. 
\end{thm}
\begin{proof}
The proof follows the exact same arguments as in the proof of Theorem~\ref{thm:isbrag_converge} with help from Lemma~\ref{lem:lyap_change_perturb} instead of Lemma~\ref{lem:lyap_change}. For the final bound, notice from~\eqref{eq:dyn_pi_hat} that $\|\Deltabpr\| \leq J$. This completes the proof.
\end{proof}

\section{When Samples are Shared}
\label{sec:dist_drones_shared}

In this section, we solve
Problem~\ref{prob}:~\eqref{prob:distributed_drones} in its full
generality, \emph{i.e.} under Assumption~\ref{asmp:shared_sample} of
shared samples. Now, the agents need to share information over the
communication network $\grph$ in order to estimate its supergradient
directions. This is in contrast to the setup in
Section~\ref{sec:dist_algo}, where every agent was capable of
computing its own supergradients as long as it had access to other's
strategies. The next result, which is a generalization of Danskin's
theorem~\cite{JMD:12}, gives a particular way of computing the
supergradients for a class of concave functions. We provide the proof
in Appendix~\ref{app:proofs}.
\begin{lem}
\label{lem:subgr_min_concave}
\thmtitle{Supergradient of the min of concave functions} Let
$\{f_i : \dom \to \real\}_{i \in \Iset}$, (where
$\dom \subseteq \real^d$ is a convex set) be a set of concave
functions. Define the concave function $g: \dom \to \real$
as 
\begin{align*}
	g(\x) \ldef \min_{i \in \Iset} f_i(\x)\,.
\end{align*}
Suppose for $\x \in \dom$, $i^\star \in \argmin_{i \in \Iset} f_i(\x)$. Then,
\begin{align}
	\del f_{i^\star}(\x) \subseteq \del g(\x)\,.
	\label{eq:intersection_subgr}
\end{align}
Further, let $f_i$ be differentiable $\forall i \in \Iset$, consider
an $\x \in \dom$, and let $\grad f_i(\x)$ be the corresponding
derivative. Then,
$\exists\, \J \subseteq \argmin_{i \in \Iset} f_i(\x)$, with
$|\J| = d+1$, such that $\forall \zetab \in \del g(\x)$,
$\zetab = \sum_{i \in \J} \lambda_i \grad f_i(\x)$, where
$\lambda_i \geq 0$, $\forall i \in \J$ and
$\sum_{i \in \J} \lambda_i = 1$.  \bulletend
\end{lem}

Thus, based on Lemma~\ref{lem:subgr_min_concave}, we argue that, in order
to compute this supergradient, it is necessary to solve the inner
optimization problem of the utility definition
in~\eqref{eq:utility_dro_shared}. This is challenging since the
optimization problem in~\eqref{eq:utility_dro_shared} is 
infinite dimensional. In order to reformulate it into a tractable
program, we use the method proposed in~\cite{PME-DK:17}. First, we
make the following assumption regarding the effect of the random
variable on an agent's utility.
\begin{assum}\thmtitle{Utility is differentiable in own strategy and
    convex in the random variable}
\label{asmp:utility_convex_rv}
For each $i \in \agt$, 
\begin{enumerate}
\item $U_i(\cdot,\s_{-i};\xi)$ is differentiable $\forall \s_{-i},\xi$, and
\item $U_i(\s;\cdot)$ is convex and differentiable $\forall \s$.
\end{enumerate}
Moreover, $\Xi$ is convex and closed.  \bulletend
\end{assum}
Now, it is possible to use the very same arguments as
in~\cite{PME-DK:17} to show that the utility defined
in~\eqref{eq:utility_dro_shared} can be rewritten as the solution of
the following finite dimensional convex program,
\begin{subequations}
\begin{align}
	\label{eq:dro_shared_opt_prob_cost_original} \U[DR](s)_i(\s) = & \!\min_{\{\x^{k}_i \in \real^r\}_{k=1}^N} \!\!\frac{1}{N}\sum_{k=1}^N  U_i \Big(\s; \h(\xih^{(k)}) - \x^{k}_i\Big)  \\
	\label{eq:dro_shared_opt_prob_bound_x} & \mathrm{s.t.} \sum_{k=1}^N \|\x^{k}_i\|_1 \leq N\epsilon_i; \\
	\label{eq:dro_shared_opt_prob_bound_xi} & \xih^{\,(k)} - \x^{k}_i \in \h(\Xi), \, \,\, \forall k \in \{1,\cdots,N\}\,.
\end{align}
\label{eq:dro_shared_opt_prob}
\end{subequations}
Note that even though this optimization problem (utility structure and
decision variables) is local to each agent; as per
Assumption~\ref{asmp:shared_sample}, agent $i \in \agt$ is only aware
of $\{\h_i(\xih^{(1)}),\cdots, \h_i(\xih^{(N)})\}$. 

To address this, we make a few simplifying assumptions. First, we
simplify the notations by letting $\hhat^k \ldef \h(\xih^{(k)})$,
$\forall k \in \{1, \cdots, N\}$. Next, we limit $\Xi$ by imposing box
constraints on the decision variables as in the following.
\begin{assum}\thmtitle{Known bounds}
\label{asmp:box_constraints}
The uncertainty set $\Xi$ is such that
$\h(\Xi) = \{ \xi \in \real^m \,|\, \bunder \leq \xi \leq \bover
\}$.The inequalities here are taken term-wise. Moreover, every agent
knows $\bunder$ and $\bover$.
\bulletend
\end{assum}
Note that this assumption on global bounds on the uncertainty could be
relaxed to agent's knowledge of component-wise bounds. If this was not
the case, agents can gather this knowledge via some max consensus
algorithm routine.  Then, notice that
since~\eqref{eq:dro_shared_opt_prob} is local to each agent
$i \in \agt$, the optimization problem remains unchanged by
summing~\eqref{eq:dro_shared_opt_prob_cost_original} (and combining
all the constraints in~\eqref{eq:dro_shared_opt_prob_bound_x}
and~\eqref{eq:dro_shared_opt_prob_bound_xi}) across all the agents. 
In what follows, we order the information available to agent
$i \in \agt$ as $[\hat{h}^k_{p_i},\cdots,\hat{h}^k_{q_i}]$, $\bunder$,
and $\bover$. Here, we choose appropriate $p_i \leq q_i \in \intpos$
to represent the starting and end indices for $i \in \agt$, which results 
into
\begin{align*}
	& \hhat^k = [\cdots,\hat{h}^k_{p_i},\cdots,\hat{h}^k_{q_i},\cdots]^\top, \quad \forall k \in \{1,\cdots,m\}\,.
\end{align*}
Then, incorporating all of these, we use the following optimization
problem to compute the supergradients,
\begin{subequations}
\begin{align}
	\label{eq:opt_problem_dist_cost} &\minimize_{\big\{\{\y_{ji}^k, z_{ji}\}_{i,j \in \agt} \big\}_{k=1}^N} \!\!\frac{1}{N} \sum_{i \in \agt} \sum_{k = 1}^N \U_i(\s_i,\s_{-i};\y_{ii}^k), \text{ s.t. } \\ 
	\label{eq:opt_problem_dist_own_inequality} & \sum_{k=1}^N \sum_{l = p_i}^{q_i} |\hat{h}_l^k - y_{ii,l}^k| + \sum_{r \in \neigh_i} [z_{ii} - z_{ir}] \leq N \epsilon_i, \quad \forall i \in \agt; \\
	\label{eq:opt_problem_dist_neigh_inequality} & \sum_{k=1}^N \sum_{l = p_i}^{q_i} |\hat{h}_l^k - y_{ji,l}^k | + \sum_{r \in \neigh_i} [z_{ji} - z_{jr}] \leq 0, \quad \forall i,j \in \agt;  \\
	\label{eq:opt_problem_dist_neigh_equality} & \y_{ji}^k - \sum_{r \in \neigh_i} [\y_{jr}^k - \y_{ji}^k] = \zero, \,\,\forall i,j \in \agt, \forall k \in \{1,\cdots,N\}; \\
	\label{eq:opt_problem_dist_box} & \bunder \leq \y_{ji}^k \leq \bover, \,\,\forall k \in \{1,\cdots,N\}, \,\, \forall i,j \in \agt\,.
\end{align}
\label{eq:opt_problem_dist}
\end{subequations}	
Here, $\hat{h}_l^k$ (resp. $y_{ij,l}^k$) are the components of
$\hhat^k$ (resp. $\y_{ij}^k$). The variable $\y_{ij}^k$
(resp. $z_{ij}$) can be thought of as agent $j$'s copy of $\y_i^k$
(similarly slack variable for $i$'s constraint). Then notice
that~\eqref{eq:opt_problem_dist} can be solved completely using local
information because of Assumption~\ref{asmp:box_constraints}.  We
state and prove our previous claim regarding the supergradients in the
next result.

\begin{lem}
  \thmtitle{Distributed optimization problem produces the required
    supergradients}
\label{lem:supergrad_opt}
Suppose
Assumptions~\ref{asmp:network_conn},~\ref{asmp:utility_convex_rv}
and~\ref{asmp:box_constraints} hold; and let
$\big\{\{\y_{ji}^{k \star}, z_{ji}^{\star}\}_{i,j \in \agt}
\big\}_{k=1}^N$ be a solution to~\eqref{eq:opt_problem_dist}. Then,
with $\yb_i^{\star} = \frac{1}{N} \sum_{k=1}^N \y_{ii}^{k\star}$,
$\forall i \in \agt$;
$\grad_{\x}[\U_i(\x,\s_{-i};\yb_i^{\star})]|_{\s_i} \in \del_{\x}
[\U[DR](s)_i(\x,\s_{-i})]|_{\s_i}$, $\forall i \in \agt$.
\end{lem} 
\begin{proof}
  First note that the description of $\U[DR](s)_i$
   in~\eqref{eq:dro_shared_opt_prob} and hence the equality between the right hand sides of~\eqref{eq:utility_dro_shared} and~\eqref{eq:dro_shared_opt_prob} 
  follows from Assumption~\ref{asmp:utility_convex_rv}
  and~\cite{PME-DK:17}.  Next it is easy to see that the optimization
  problem in~\eqref{eq:dro_shared_opt_prob} (combined across all
  agents) is equivalent to the following problem,
\begin{subequations}
\begin{align}
	\label{eq:dro_shared_opt_prob_cost} & \minimize_{\{\{\y_i^k\}_{i \in \agt}\}_{k=1}^N} \frac{1}{N} \sum_{i \in \agt} \sum_{k = 1}^N \U_i(\s_i,\s_{-i};\y_i^k) \\
	\label{eq:dro_shared_opt_prob_bound_one_norm} & \quad \mathrm{s.t.} \sum_{k=1}^N \sum_{l = 1}^m |\hat{h}_l^k - y_{i,l}^k | \leq N \epsilon_i, \,\, \forall i \in \agt \\
	\label{eq:dro_shared_opt_prob_bound_xi_combined} & \qquad \,\,\, \bunder \leq \y_i^k \leq \bover, \,\,\forall k \in \{1,\cdots,N\}, \,\, \forall i \in \agt\,.
\end{align}	  
\label{eq:dro_shared_opt_prob_combined}
\end{subequations}
This is because, we have used the auxiliary variables $\y_i^k$ to
satisfy the equality constraint $\y_i^k = \hhat^k - \x_i^k$,
$\forall \,i,k$. Then,~\eqref{eq:dro_shared_opt_prob_bound_one_norm}
is the same as~\eqref{eq:dro_shared_opt_prob_bound_x} due to the
definition of $\|\cdot\|_1$. Moreover,
~\eqref{eq:dro_shared_opt_prob_bound_xi_combined} is the same
as~\eqref{eq:dro_shared_opt_prob_bound_xi} due to
Assumption~\ref{asmp:box_constraints}. Thus, optimization
problems~\eqref{eq:dro_shared_opt_prob}
and~\eqref{eq:dro_shared_opt_prob_combined} are equivalent.

Next we introduce auxiliery variables to show
that~\eqref{eq:dro_shared_opt_prob_combined}
and~\eqref{eq:opt_problem_dist} are equivalent. To do this, we first
label the feasible sets of~\eqref{eq:dro_shared_opt_prob_combined}
and~\eqref{eq:opt_problem_dist} 
$\Fset$ and $\Fsetb$ respectively; \emph{i.e.}
$\{\{\y_i^k\}_{i \in \agt}\}_{k=1}^N \in \Fset$ and
$\{\{\y_{ji}^k, z_{ji}\}_{i,j \in \agt} \}_{k=1}^N \in \Fsetb$. Now,
since the cost functions in~\eqref{eq:dro_shared_opt_prob_cost}
and~\eqref{eq:opt_problem_dist_cost} are the same, it is enough to
show the following two claims to show equivalence.

\emph{Claim i:}
$\forall \{\{\y_{ji}^k, z_{ji}\}_{i,j \in \agt} \}_{k=1}^N \in
\Fsetb$, $\{\{\y_{ii}^k\}_{i \in \agt}\}_{k=1}^N \in \Fset$. \\To show
this, consider a
$\{\{\y_{ji}^k, z_{ji}\}_{i,j \in \agt} \}_{k=1}^N \in \Fsetb$. Since
$\{\{\y_{ii}^k\}_{i \in \agt}\}_{k=1}^N$
satisfies~\eqref{eq:opt_problem_dist_box}, it also
satisfies~\eqref{eq:dro_shared_opt_prob_bound_xi_combined}. Moreover,
because of~\eqref{eq:opt_problem_dist_neigh_equality}, and
Assumption~\ref{asmp:network_conn} on graph connectivity,
$\y_{ii}^k = \y_{ij}^k$, $\forall i,j \in \agt$. We use
Assumption~\ref{asmp:network_conn} in the next argument
also. Summing~\eqref{eq:opt_problem_dist_own_inequality} for
$i \in \agt$ and~\eqref{eq:opt_problem_dist_neigh_inequality}
$\forall j \in \agt \setminus \{i\}$ shows that
$\{\{\y_{ii}^k\}_{i \in \agt}\}_{k=1}^N$
satisfies~\eqref{eq:dro_shared_opt_prob_bound_one_norm} because of the
previous equality. This concludes the proof of this claim.

\emph{Claim ii:} $\forall \{\{\y_i^k\}_{i \in \agt}\}_{k=1}^N \in \Fset$, $\exists \, \{z_{ji}\}_{i,j \in \agt}$ such that $\{\{\y_{ji}^k, z_{ji}\}_{i,j \in \agt} \}_{k=1}^N \in \Fsetb$ with $\y_i^k = \y_{ij}^k$, $\forall i,j \in \agt$.\\
To show this, consider a
$\{\{\y_i^k\}_{i \in \agt}\}_{k=1}^N \in \Fset$ and take vectors
$\{\{\y_{ji}^k\}_{i,j \in \agt} \}_{k=1}^N$ such that
$\y_i^k = \y_{ij}^k$, $\forall i,j \in \agt$. Then, because
of~\eqref{eq:dro_shared_opt_prob_bound_xi_combined} and
Assumption~\ref{asmp:network_conn}, by construction,
$\{\{\y_{ji}^k\}_{i,j \in \agt} \}_{k=1}^N$
satisfies~\eqref{eq:opt_problem_dist_neigh_equality}
and~\eqref{eq:opt_problem_dist_box}.
Then the claim is easy to show because of the connectivity assumption.
This concludes the proof of this other claim.

Combining these concludes that
~\eqref{eq:dro_shared_opt_prob_combined}
and~\eqref{eq:opt_problem_dist} are equivalent.  Then the proof of the
claim in this lemma follows from
Lemma~\ref{lem:subgr_min_concave}. The proof is now complete.
\end{proof}

We would like to point out that even though the technique we used in
the previous proof is similar to~\cite{AC-JC:16-allerton}; our setup
is very different, since the agents do not have access to the cost
function and constraints simultaneously.  We would like to remind the
reader that~\eqref{eq:opt_problem_dist} is a convex program that can
be solved entirely using local information. As such, agent
$i \in \agt$ handles the variables
$\{\{\y_{ji}^k, z_{ji}\}_{j \in \agt}\}_{k=1}^N$. There is a whole
body of literature that provides iterative solutions
to~\eqref{eq:opt_problem_dist}. In this work, we use the primal-dual
subgradient method~\cite{SB-LV:04} using the augmented Lagrangian
approach to compute the required supergradients for the d-\algoname
updates.
While it is possible to provide ISS properties for the distributed
updates and perform the updates of such an algorithm in a synchronous
manner with~\eqref{eq:disbrag} we reserve that analysis for future
work. To that extent, we condense the d-\algoname algorithm for shared
samples into Algorithm~\ref{alg:disbrag}. We formally state the
convergence properties of this algorithm in the next result and
conclude by remarking on the communication protocol later. Since this
a direct consequence of all the previous results, we skip a formal
proof.

\begin{thm}\thmtitle{Convergence of d-\algoname with shared samples}
\label{thm:final_result}
Suppose the hypothesis of Theorem~\ref{thm:disbrag_converge} holds and
suppose $\s(t)$ is the sequence obtained from
Algorithm~\ref{alg:disbrag} from initial condition
$\{\s_i(0) = \s_i(-1) \in \set_i\}_{i \in \agt}$. Then, the bound
in~\eqref{eq:dyn_iss_w_consensus} holds with
$\omega(\cdot) \ldef d(\cdot,\ne(\game[DR](s)))$.  \proofend
\end{thm}

\begin{algorithm}[t]
\caption{d-\algoname for shared samples}
\label{alg:disbrag}
\begin{algorithmic}[1]
\Require Consensus time steps $\Tcon$, optimization time steps $\Topt$

\State Initialize $\s_i \gets \s_i(0) \in \set_i$, $\forall i \in \agt$
\State Initialize $\sh_{ij} \gets \texttt{random}(\set_j)$, $\forall i \in \agt$, $\forall j \in \agt \setminus \{i\}$ \LineComment{randomly initialize estimates}

\For{$t \in \intnonneg$}
	\For{$\tau \in {1,\cdots,\Tcon}$}
		\State Update $\{\sh_{ij}\}_{j \in \agt \setminus \{i\}}$ using steps~\eqref{eq:disbrag_chi_ij},~\eqref{eq:disbrag_chi_ii} of dynamic consensus 
	\EndFor
	
	\For{$t' \in {1,\cdots,\Topt}$}
		\State Compute $\{\{\y_{ji}^{k \star}, z_{ji}^{\star}\}_{i,j \in \agt}\}_{k=1}^N$ that is a solution to~\eqref{eq:opt_problem_dist} in a distributed way
	\EndFor
	
	\For{$i \in \agt$}
		\State $\yb_i^{\star} \gets \frac{1}{N} \sum_{k=1}^N \y_{ii}^{k\star}$
		\State $\phibh_i \gets \mu_i\grad_{\x}[\U_i(\x,\s_{-i};\yb_i^{\star})]|_{\s_i}  + \w_i(\s_i,t)$
		\State $\s_i \gets \s_i + \alpha_i \, \Big[\suppf_i(\phibh_i) - \s_i \Big]$
	\EndFor
\EndFor
\end{algorithmic}
\end{algorithm}

\section{Simulations} \label{sec:sims} In this section, we illustrate
the trajectory evolution of the algorithms that we proposed. For all
our simulations we take $6$ agents, \emph{i.e.}
$\agt = \{1,\cdots,6\}$. All values are rounded off to 3 decimal places.

\paragraph*{Effect of step-size $\alpha_i$:} First we study the effect
of $\{\alpha_i\}_{i \in \agt}$ on the convergence of \algoname in
Definition~\ref{def:disbrag}. We specifically want to show that our
algorithm is able to converge near a NE even when the utilities are
non-smooth. To that affect, we choose $\set_i = [0,2]$,
$\forall i \in \agt$. Then, we set the utilities of the agents as
\begin{align*}
  \U[DR]_i(s_i,s_{-i}) = - |s_i - 0.25 \times i| \times \prod_{j \in \agt \setminus \{i\}}s_j, \quad \forall i \in \agt.
\end{align*}
Note that we have denoted the strategies as $s_i \in \real$. Then
clearly the strategy profile $s_1^\star = 0.25$, $s_2^\star = 0.5$,
$s_3^\star = 0.75$, $s_4^\star = 1$, $s_5^\star = 1.25$, and
$s_6^\star = 1.5$ is the unique NE. We simulate \algoname from an
initial condition $s_1(0) = 0.097$, $s_2(0) = 0.578$,
$s_3(0) = 1.442$, $s_4(0) = 0.043$, 
$s_5(0) = 0.412$, and $s_6(0) = 0.101$. Moreover,
to isolate the effect of the step-size parameter, we fix
$\mu_i = 0.5$, and $\lambda_i = 1$, $\forall i \in \agt$. The solution
trajectories for two cases: i) $\alpha_i = 0.1$, $\forall i \in \agt$,
and ii) $\alpha_i = 0.01$, $\forall i \in \agt$ are shown in
Figure~\ref{fig:alpha_effect}. In accordance to the bound provided in
Theorem~\ref{thm:isbrag_converge}, notice that with the higher the
$\alpha_i$ value, the solutions converge faster, but they converge to a
larger set around the NE. On the other hand, with a smaller $\alpha_i$
value, the solutions converge closer to the NE, but take more
time-steps to get there.

\begin{figure}
\begin{tabular}{c}
	\qquad \,\, \begin{tikzpicture}
		\node[draw, rounded corners=3pt, inner sep=7pt] {
    			\includegraphics[trim=40pt 17pt 40pt 15pt, clip, scale=1]{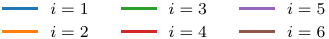}
		};
	\end{tikzpicture} \\
	\includegraphics[trim=7pt 20pt 7pt 5pt, clip, scale=1]{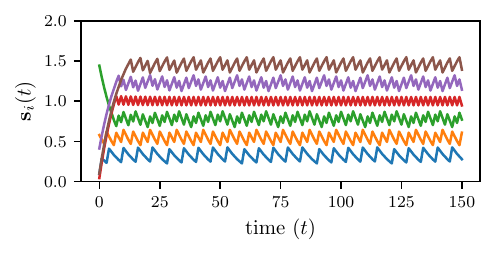} \\
	\includegraphics[trim=7pt 7pt 7pt 5pt, clip,scale=1]{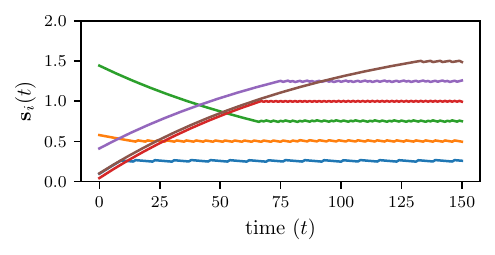}
\end{tabular}
\caption{Effect of step-size parameters $\{\alpha_i\}_{i \in \agt}$ on convergence of \algoname. The plots share a common legend. (Top) $\alpha_i = 0.1$, $\forall i \in \agt$. (Bottom) $\alpha_i = 0.01$, $\forall i \in \agt$.}
\label{fig:alpha_effect}
\end{figure}

\begin{figure}
\begin{tabular}{c}
	\qquad \,\,\, \begin{tikzpicture}
		\node[draw, rounded corners=3pt, inner sep=7pt] {
    			\includegraphics[trim=40pt 17pt 40pt 15pt, clip, scale=1]{legend.pdf}
		};
	\end{tikzpicture} \\
	\includegraphics[trim=7pt 20pt 7pt 5pt, clip, scale=1]{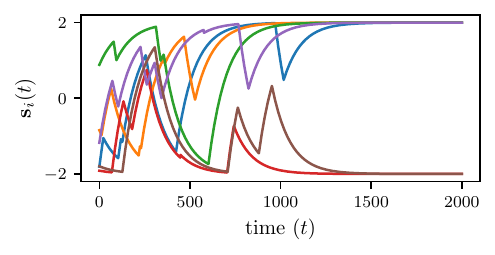}\\
	\includegraphics[trim=7pt 20pt 7pt 5pt, clip, scale=1]{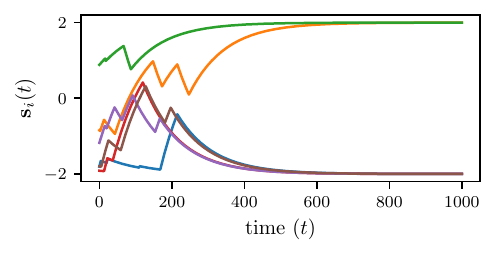}\\
	\includegraphics[trim=7pt 20pt 7pt 5pt, clip, scale=1]{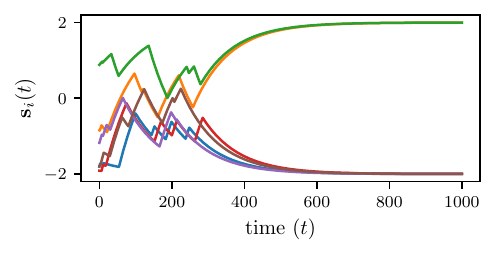}\\
	\includegraphics[trim=7pt 7pt 7pt 5pt, clip, scale=1]{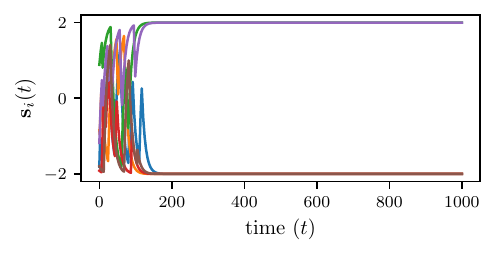}\\
\end{tabular}
\caption{Convergence properties of d-\algoname. The plots share a common legend Different plots correspond to different values of communication time steps $\Tcon$ and step-size parameters $\{\alpha_i\}_{i \in \agt}$. The figures from the top are arranged as follows. (First) $\Tcon = 10$, $\alpha_i = 0.01$, $\forall i \in \agt$. (Second) $\Tcon = 50$, $\alpha_i = 0.01$, $\forall i \in \agt$. (Third) $\Tcon = 100$, $\alpha_i = 0.01$, $\forall i \in \agt$. (Fourth) $\Tcon = 100$, $\alpha_i = 0.1$, $\forall i \in \agt$.}
\label{fig:disbrag}
\end{figure}

\paragraph*{Performance of d-\algoname:}
For this scenario, we test our algorithm on a non-monotone game and
report (empirically) on how the parameters of the game affect the
steady-state.  Further, we allow the agents to communicate over a
cyclic graph only, \emph{i.e.}
$\edg = \{(1,2), (2,3), (3,4), (4,5), (5,6), (6,1)\}$. We choose
$\set_i = [-2,2]$, $\forall i \in \agt$ and set the utilities of the
agents as
\begin{align*}
  \U[DR](s)_i(s_i,s_{-i}) = \prod_{j \in \agt} s_j, \quad \forall i \in \agt.
\end{align*}
It is easy to see that the Nash equilibrium is not
unique. Specifically, any NE (other than $\sst = \zero$) must have
$s_i^{\star} \in \{-2,2\}$, and there must be even number of agents
with strategy $-2$. The evolutions of d-\algoname, from an initial
strategy profile of $s_1(0) = -1.805$, $s_2(0) = -0.844$,
$s_3(0) = 0.884$, $s_4(0) = -1.913$,
$s_5(0) = -1.176$, and $s_6(0) = -1.797$, 
with different choices of $\Tcon$ and
$\{\alpha_i\}_{i \in \agt}$ are shown in Figure~\ref{fig:disbrag}.
Note that, the strategy we converge to is dependent on several of the
algorithm parameters. Firstly, the first three figures from the top
have same step-size parameters and differ in the number of consensus
steps ($\Tcon$). As expected, the time steps needed to converge
decreases as $\Tcon$ increases. The last two graphs in
Figure~\ref{fig:disbrag} share the same value of $\Tcon$ and differ in
the step-size parameters. The larger the $\alpha_i$, the faster the
convergence, but at the cost of high chattering-like behavior in the
transients; while the smaller the $\alpha_i$ the longer it takes to
converge with a reduced chattering in the transients. In any case,
these trajectories do not sustain limiting oscillatory behavior (as
opposed to the trajectories in Figure~\ref{fig:alpha_effect}) since
the utilities are smooth and the game admits proper pseudogradients.

\section{Conclusion} \label{sec:conclude} In this work, we provide
centralized and distributed algorithms that seek the Nash equilibria
of a distributionally robust game (DRoNE's). First, we present
conditions under which these DRoNE's exist, and relate them to the NE
of the associated stochastic game. This relation (which depends on the
number of sample points available to the agents) predicts that as the
number of data points increase, with high probability DRoNE's become
better approximations of the NE's of the stochastic game. Second, we
define and provide conditions on the parameters of the \algoname
dynamics that ensure convergence to arbitrarily small regions around
the DRoNE's. This comes at a cost of slower rate of convergence. To
handle the distributed case, we provide methods in which the agents
can approximate the supergradients of their (possibly) non-smooth
utility by exchanging information between their neighbors. When the
game has amicable supergradients, we introduce an itertial term that
produces monotonicity-like behaviors. When the game has non-unique NE,
the region around which DRoNE the algorithm converges to is highly
dependent on the choice of algorithm parameters.  In the future, we
would like to extend the class of games compatible with this framework
by introducing more general inertial terms. We will also handle
simultaneous
  communication
  and estimate update among agents.

  \section*{Declaration of generative AI and AI-assisted technologies
    in the manuscript preparation process}
  During the preparation of this work the authors used ChatGPT to help
  with the visualization of the solution trajectories and figures in
  Section~\ref{sec:sims}. After using this tool/service, the authors
  reviewed and edited the content as needed and take full
  responsibility for the content of the published article. No AI
  tool/service was used for preparing the code that generates the
  solution trajectories or for any other part of the paper.

\appendix
\section{Auxiliary proofs} \label{app:proofs}

\paragraph*{Proof of Lemma~\ref{lem:change_max_lin_func}:}
Consider an arbitrary but fixed $\phib_i^1$, with
$\|\phib_i^1\| \leq M_i$ and an arbitrary but fixed
$\x_i^* \in \Xset^*_i(\phib_i^1)$. Then
$f_i(\phib_i^1) = \ip{\x_i^*}{\phib_i^1}$. Further, since $\|\cdot\|$
is non-negative,
\begin{align*}
  \ip{\x_i^*}{\phib_i} \leq
  & \,\, \ip{\x_i^*}{\phib_i^1} + \ip{\x_i^*}{[\phib_i - \phib_i^1]} + \mathfrak{C} \|\phib_i - \phib_i^1\|, 
\end{align*}
$\forall \phib_i \in \ball_{M_i}(\zero)$, for any
$\mathfrak{C} \in \realnonneg$. Now take an arbitrary but fixed
$\phib_i^2 \in \ball_{M_i}(\zero)$ and let
$\y_i^* \in \Xset^*_i(\phib_i^2)$. Thus, 
$f_i(\phib_i^2) = \ip{\y_i^*}{\phib_i^2}$. Then, using the previous
inequality,
\begin{align*}
  \ip{\y_i^*}{\phib_i^2} & \leq \ip{\x_i^*}{\phib_i^1} + \ip{\x_i^*}{[\phib_i^2 - \phib_i^1]} + \mathfrak{C} \|\phib_i^2 - \phib_i^1\| \\
                         & \hspace{3em} + \ip{[\y_i^* - \x_i^*]}{[\phib_i^2 - \phib_i^1]} + \ip{[\y_i^* - \x_i^*]}{\phib_i^1} \\
                         & \leq \ip{\x_i^*}{\phib_i^1} + \ip{\x_i^*}{[\phib_i^2 - \phib_i^1]} + [\mathfrak{C} + D_i] \|\phib_i^2 - \phib_i^1\|\,.
\end{align*}
For the last inequality, we rely on the Cauchy-Schwarz inequality,
$|\ip{[\y_i^* - \x_i^*]}{[\phib_i^2 - \phib_i^1]}| \leq \|\y_i^* -
\x_i^*\| \|\phib_i^2 - \phib_i^1\|$, and further used that $D_i$ is
the diameter of $\set_i$. Moreover, by definition
$\ip{[\y_i^* - \x_i^*]}{\phib_i^1} = \ip{\y_i^*}{\phib_i^1} -
\ip{\x_i^*}{\phib_i^1} \leq 0$, since $\y_i^* \in \set_i$. Now, since
$\mathfrak{C} + D_i \geq D_i$ is independent of $\phib_i^1$,
$\phib_i^2$, and $\x_i^*$, the proof is complete.  \proofend

\paragraph*{Proof of Lemma~\ref{lem:subgr_min_concave}:}
We first prove the first claim. Let $\x \in \dom$, $i^\star \in \Iset$
be as in the hypothesis and let $\zetab \in \del
f_{i^\star}(\x)$. Then, $\forall \y \in \set$,
\begin{align*}
  g(\y) \leq f_{i^\star}(\y) \leq f_{i^\star}(\x) + \ip{\zetab}{[\y-\x]} = g(\x) + \ip{\zetab}{[\y-\x]}\,.
\end{align*}
Here, the first inequality and the equality comes from the definition
of $g$ and the assumption on $i^\star$. The second inequality comes
from Definition~\ref{def:subgr}. This
proves~\eqref{eq:intersection_subgr}.

Now, for the second part, by property of the supergradients and the hypothesis, $\del g(\x) = \text{co} (\{\nabla f_i(\x)\}_{i \in \Iset})$, where $\text{co}(\cdot)$ represents the convex hull.
Then, the second part follows directly from Carath{\'e}odory's
theorem from convex geometry. The proof is now complete.  \proofend

%

\bibliographystyle{unsrt}      
\bibliography{./bib/alias.bib,./bib/SM.bib,./bib/JC.bib,./bib/SMD-add.bib}    

\end{document}